\documentclass[12pt]{amsart}
\usepackage[utf8]{inputenc}
\usepackage{amsthm}
\usepackage{amsfonts}
\usepackage{comment}
\usepackage{amssymb}
\usepackage{eufrak}
\usepackage{comment}
\usepackage{hyperref}
\usepackage{mathtools}
\usepackage{xcolor}
\usepackage{tikz-cd}
\usepackage{tkz-euclide}
\usepackage{pgfplots}
\usepackage{float}
\usepackage{enumerate}

\usepackage{breqn} 

\makeatother
\newtheorem{Theorem}{Theorem}
\newtheorem{Corollary}[Theorem]{Corollary}
\newtheorem{Lemma}[Theorem]{Lemma}
\newtheorem{Proposition}[Theorem]{Proposition}
\newtheorem{Conjecture}[Theorem]{Conjecture}

\theoremstyle{remark}
\newtheorem{remark}[Theorem]{Remark}

\counterwithin{Theorem}{section}
\counterwithin{equation}{section}
\usepackage[margin=1in]{geometry}
\usepackage{cite}

\def\F{\mathbb{F}}

\def\Q{\mathbb{Q}}
\def\R{\mathbb{R}}

\def\Z{\mathbb{Z}}

\title{Degrees of points on irreducible hypersurfaces}
\author{Lea Beneish}
\author{Andrew Granville}

\begin{document}

\begin{abstract} 
We study the set of $D$ such that a given irreducible hypersurface $C$ of degree $d$ has infinitely many points of degree $D$ over $\mathbb{Q}$. We give a new explicit proof that this set contains all (positive) multiples of the index of $C$ with finitely many exceptions. When $D$ is sufficiently large and divisible by the index of $C$, we show there are $\gg x^{1/2d^2-\epsilon}$ distinct fields with degree $D$ and discriminant $\leq x$ containing new non-singular points on $C$. Our proof relies on (what we define to be) the index of the Newton polytope $H(C)$ for $C$ which we use as combinatorial proxy for the index of $C$. We conjecture that for almost all $C$ with a given Newton polytope $H$, the index of $H$ equals the index of $C$ and we prove this conjecture for a positive proportion of curves with  $H(C)=H$.  As an application of our techniques, we prove half of Bhargava's conjecture on the least odd degree of points on a typical hyperelliptic and we recover Springer's theorem and a related statement for rational points on cubic hypersurfaces.
\end{abstract}
\maketitle

\section{Introduction}
Understanding and quantifying algebraic points on curves, surfaces and more generally higher dimensional varieties, is a fundamental subject of study in Diophantine geometry. Here we develop a relationship between the degrees of non-singular algebraic points on a given hypersurface and its ``shape''.

Let $C=C_f$ be an irreducible hypersurface given by the projective equation 
\[
f(x_1,\dots,x_m)=0
\]
(so that $m\geq 3$) defined over a number field $K$.  
Let $\mathcal F(C)$ denote the set of  fields $L$ for which $C$ has \emph{new points}, that is, non-singular  $L$-rational points that do not belong to a subfield of $L$, and define 
\[ D_{K}(C):= \{ n\geq 1:\ \exists\ L\in \mathcal F(C) \text{ with } [L:K]=n \}, \]
the set of degrees of the fields in $\mathcal F(C)$ over $K$.
We also define $D_{K}(C)^{\text{\rm inf}}$ to be the set of integers $n$ for which there are infinitely many fields $L$ in $\mathcal F(C)$ of degree $n$ over $K$.

 When we have a fixed number field $K$ we will omit the subscript $K$ and write $D(C)$. 
 Throughout we will take $K=\Q$ for simplicity though our results should hold for all number fields $K$ via minor and expected modifications of the same proofs.
 Let 
\[
G(C):=\gcd_{n\in D(C)} n
\]
be the \emph{index} of $C$,\footnote{The   index is often   defined to be the least positive degree of a $K$-rational divisor on $C$, though it is well-known that these definitions are equivalent.}
 so that  if $n\in D(C)$ then $G(C)$ divides $n$. The converse is almost true:

\begin{Theorem} \label{thm2} There exists a finite set of integers $ \mathcal E(C)$ such that
\[
D(C)^{\text{\rm inf}} = G(C)\mathbb{Z}_{\geq 1} \setminus \mathcal E(C)
\]
\end{Theorem}

In other words, the set of exceptions,
  \[
\mathcal E(C):=\{ n\geq 1: G(C) \text{ divides } n \text{ and } n\not\in D(C)^{\text{\rm inf}}\},
\]
is  finite  and so if $n$ is sufficiently large then $n \in D(C)^{\textrm{inf}}$ if and only if $G(C)$ divides $n$.

 Clark, Milosevic and  Pollack \cite[Theorem 2.4]{CMP} proved that the set of degrees for which there are new points is of the form $G(C)\mathbb{Z}_{\geq 1} \setminus \mathcal E(C)$   for  any ``nice''  variety of positive dimension over  a Hilbertian field $K$ of characteristic zero. One can then deduce (the stronger) Theorem \autoref{thm2} for curves of genus $>1$ by an application of Faltings' Theorem.

Our proof of Theorem \autoref{thm2} is for all irreducible hypersurfaces and  gives an explicit lower bound on the number of fields $L\in \mathcal F(C)$ with $[L:K]=n$ of discriminant $\leq x$ for each sufficiently large $n\in D(C)^{\text{\rm inf}}$.
 
The basic idea of both proofs is linear algebra, in ours (see section 4) we explicitly construct solutions whereas   \cite[Theorem 2.4]{CMP}  used the Riemann-Roch theorem to show   the existence of points of these degrees.   Our  proof of Theorem \ref{thm2} allows us to give an explicit lower bound on the number of fields generated by points on hypersurfaces and in the case of curves, to be explicit about what $\mathcal E (C)$ is.

\subsection{General quantitative results}

  For those degrees for which there are infinitely many points of degree $D$, we quantify them as suggested by  Mazur, Rubin and Larsen in \cite{MRL} by giving a lower bound on the number of fields generated by points of degree $D$.

\begin{Theorem}\label{thm: the count}  
Let $C$ be an irreducible hypersurface given by a polynomial  $f$ in $m$ variables with integer coefficients, of degree $d$.
Fix $\epsilon>0$. If $D$ is divisible by $G(C)$ and is sufficiently large, and $X$ is sufficiently large then 
\[
\#\{ \text{\rm Fields } K\in \mathcal F(C): [K:\mathbb Q]=D \text{ and {\rm  disc}}(K/\mathbb Q)  \leq X\} \gg X^{\frac 1{2d^2}-\epsilon}.
\]
\end{Theorem}

In Theorem \ref{thm: the count revisited}  we will prove a stronger result with ``$\frac 1{2d^2}$'' replaced by some exponent $\text{\rm Exp}(C)$, and we prove that 
$\text{\rm Exp}(C)\geq \max_j \frac 1{2(\deg_{x_j}f)^2} \geq \frac 1{2d^2}$.

 The best upper bound known \cite{OT22}  on the number of field extensions of degree $D$ is 
 \[
\#\{ \text{Field extensions } K/L: [K:L]=D \text{ and disc}(K/L)  \leq X\} \ll_{L,D} X^{2 (\log D)^2}
\]
 (where the constant ``2'' can be improved); it is believed that the correct quantity is $ \sim c_{L,D} X$.

 We can be  more explicit about $D(C_f)^{\text{\rm inf}}$ when there is a non-singular rational point on $C$.
 The \emph{Frobenius set} generated by given positive integers $a_1,\dots,a_k$ is
 \[ 
 \text{Frob}(a_1,\dots,a_k):=\{ a_1n_1+\cdots+a_kn_k: n_1,\dots,n_k\in \Z_{\geq 0}\}.
 \]
Then we have the following:
  
\begin{Theorem}\label{thm: A rational point} 
Let $C$ be an irreducible hypersurface given by a polynomial in $m$ variables with integer coefficients, of degree $d$
and suppose that $C$ has a non-singular rational point. Then  $G(C_f)=1$ and $D(C_f)^{\text{\rm inf}} \supset d^2+ \text{\rm Frob}(d-1,d)$ so that $\mathcal E(C_f)\subset \{1,\dots, 2d^2-3d+1\}$. 
\end{Theorem}  

\begin{remark} \label{rem: genus g}
If $C$ is a curve with no singular points then $C$ has genus $g=\frac{(d-1)(d-2)}2$. 
Theorem \ref{thm: A rational point}  implies that if $C$ has a rational point then there are infinitely many (non-singular) points of degree $D$ on $C$, for every $D\geq 2g+d^2$.
\end{remark}

\subsection{A combinatorial proxy for the index} 
For a given  irreducible hypersurface $C_f$  it is difficult to determine the index explicitly.
If $C_f$ is given by the projective equation  
\begin{equation}\label{eq: f written out}
f(x_1,\dots,x_m)=\sum_{(i_1,\dots,i_m)\in I_f} c_{i_1,\dots,i_m}  x_1^{i_1}\cdots x_m^{i_m}
\end{equation}
and if $d$ is the degree of $f$ then 
\[
I_f\subset \mathcal M_{m,d}:=\{ \mathbf{i}=(i_1,\dots,i_m)\in \mathbb Z_{\geq 0}^m: i_1+\cdots+i_m=d\},
\]
and for each $j, 1\leq j\leq n$ there  exists $\mathbf{i}\in I_f$ with $i_j=0$.
Let $H_f$ be the \emph{Newton polytope} of $f$, which is the set of corners of $I_f$, where we define $\mathbf{h}\in I_f$ to be a \emph{corner} of $I_f$
if there exists $\mathbf{v}\in \mathbb Z_{\geq 0}^m$ for which
\[
\mathbf{h} \cdot \mathbf{v} > \mathbf{i} \cdot \mathbf{v} \text{ for all } \mathbf{i}\in I_f, \mathbf{i}\ne \mathbf{h}.
\]
For each $\mathbf{h}=(h_1,\dots,h_m) \in \mathcal M_{m,d}$ define $g(\mathbf{h}):=\gcd(h_1,\dots,h_m)$.

Let $\mathcal H_{m,d}$ be the set of non-empty subsets $H$ of $\mathcal M_{m,d}$ for which every element of $H$ is a corner of $H$ and, for each $j=1,\dots,m$, there exist  $\mathbf{h}, \mathbf{i}\in H$ with $h_j=0<i_j$. For each 
$H\subset  \mathcal H_{m,d}$, define 
 \[
 G( H):= \gcd_{ \mathbf{h}\in   H}    g(\mathbf{h});
 \]
 in other words, $G(H)$ is the gcd of all of the co-ordinates of all the elements of $H$.
 In section \ref{sec: Construct D(H)inf}  we will define a finite set $\mathcal E(H)$ of integer multiples of $G(H)$.

 We begin by showing how $ D(C_f)^{\text{\rm inf}}$  is related to $D(H_f)^{\text{\rm inf}}$:

  \begin{Theorem} \label{thm: C_f and H+} For any given $f(x_1,\dots,x_m)\in \mathbb Z[x_1,\dots,x_m]$ we have
  \[
  G(C_f) \text{ divides } G(H_f);
  \]
  and if  $G(C_f)=G(H_f)$ then $ \mathcal E(C_f)\subset \mathcal E(H_f)$.
 \end{Theorem}

\subsection{A conjecture: More than a  proxy} 
 
Our main claim is that  $D(C)^{\text{\rm inf}}$ usually depends \emph{only} on its Newton polytope, $H_f$:
 
 \begin{Conjecture} For almost all $f$ with $H_f=H$ (where the $f$ are ordered by the size of their coefficients, which \emph{all} vary) we have
 \[
 G(C_f)=G(H) \text{ and }  \mathcal E(C_f)=\mathcal E(H).
 \]
 In other words,
 $D(C_f)^{\text{\rm inf}}$ is typically \emph{entirely determined} by $H_f$.
 \end{Conjecture}

 We will prove this conjecture for (twisted) hyperelliptic curves $by^2=g(x)$ of even degree $d$.  
This curve can be written as  $by^2=\sum_{i=0}^d g_ix^i$ with $g_d\ne 0$ in  affine form  and therefore homogenizes to
 \[
f_g(x,y,z):=\sum_{i=0}^d g_ix^iz^{d-i} -  by^2z^{d-2} =0
\]
so that $I_f \subset I_d:= \{ (0,2,d-2)\} \cup \{ (i,0,d-i): 0\leq i\leq d)\}$. If $g(0)\ne 0$ then
 $ H_f= H:=\{ (0,2,d-2),(0,0,d), (d,0,0)\}$ so that $G(H_f)=2$. If   $g(0)= 0$ and $x^2\nmid g(x)$ (so that the point $(0,0)$ is non-singular) then evidently $G(C_f)=1$.

 \begin{Theorem} \label{thm: Hyperelliptic}
 Fix even integer $d\geq 2$ and $H:=\{ (0,2,d-2),(0,0,d), (d,0,0)\}$.
 Let $C_{f_g}$ be the hyperelliptic curve $by^2=g(x)$ where $g(x)\in \mathbb Z[x]$ has degree $d$.
\begin{enumerate}[{\rm (a)}]
\item For almost all $g(x)\in \mathbb Z[x]$ of even degree $d$ with $g(0)\ne 0$  we have $G(C_{f_g})=G(H)=2$ and $D(C_f)^{\text{\rm inf}}=  2\Z_{\geq 1}$.  
\item If there is a non-singular rational point on $C_f$ then $D(C_f)^{\text{\rm inf}}=  \Z_{\geq 1} \setminus \mathcal E(C_f)$ where $\mathcal E(C_f)\subset  \{ 1,3,5\dots, d-3\}$.
 \end{enumerate}
 \end{Theorem}
 
Since $g(0)\ne 0$ for almost all polynomials $g(x)$, we deduce from (a) that almost all hyperelliptic curves $by^2=g(x)$ have infinitely many algebraic points of every even degree, but no algebraic points of odd degree. In particular they have no rational points.
 
 If $g(x)$ has odd degree $d$ then let $G(t)=t^{d+1}g(1/t)$, which has (even) degree $d+1$, and then there is a bi-rational transformation between the algebraic points of $y^2=g(x)$ and those of $Y^2=G(X)$. Now $Y^2=G(X)$ has the non-singular rational point at $(0,0)$, and so we apply Theorem \ref {thm: Hyperelliptic}(b).

We prove  Theorem \ref {thm: Hyperelliptic}(a)    in section \ref{sec: Hyper index 2},
and Theorem \ref {thm: Hyperelliptic}(b)    in section \ref{sec: SuperHype}   as well as  some generalizations.

 Our conjecture holds for a positive proportion of curves $f$ with $I_f=I$, for most $I$:
 
  \begin{Theorem} \label{thm: PosProp} Fix integer $m\geq 1$ and  
  assume that $\{ (i,j,0): i+j=m\} \cup \{ (0,0,m)\} \subset I\subset \mathcal M_{3,d}$; in all cases
  $H(I)=H:=\{ (0,0,m), (0,m,0), (m,0,0)\}$.  For at least a positive proportion of  $f(x,y)\in \Z[x,y]$ of degree $m$ with $I_f=I$ we have $G(C_{f})=m\Z_{\geq 1}=G(H)$. 
 \end{Theorem}

 \subsection{Applications to the sparsity of odd degree points}
Finally, we give applications of our methods to quadric and cubic hypersurfaces and to hyperelliptic curves. 
 
 \subsubsection{Applications to hypersurfaces}
 Our methods can be used to show in certain cases that higher degree points on a hypersurface imply the existence of rational points. In particular we reprove Springer's theorem \cite{Springer} and a result of Coray \cite{Coray}.
 
 \begin{Theorem}[Springer]\label{thm: Springer} Any homogeneous quadratic $f(x_1, \dots, x_m)=0$  has a rational point if and only if it has a point of odd degree.
\end{Theorem} 

We also have the following statement for cubic forms relating rational points and degree $2$ points.

\begin{Proposition}[Coray]\label{prop: cubics} A homogeneous cubic  $f(x_1, \dots, x_m)=0$ has a rational point if and only if it has a point of degree $2$.
\end{Proposition}

 \subsubsection{Applications to hyperelliptic curves}
 Bhargava, Gross, and Wang \cite[Corollary 8]{BGW}  showed that for any fixed odd integer $m$, the proportion of hyperelliptic curves of genus $g$ that contain rational points of degree $m$, tends to $0$ as $g\to \infty$. However, in the cases where a hyperelliptic curve does possess an odd degree point, Bhargava conjectures that $100\%$ of the locally soluble hyperelliptic curves which have odd degree points will have an odd degree point of lowest degree \emph{exactly} $g$ or $g+1$ (whichever is odd). Below, we give a characterization of $D(C)^{\text{\rm inf}}$ in the case where a hyperelliptic curve has an odd degree point.

  \begin{Proposition} \label{Prop: HyperI} Let $C$ be a hyperelliptic curve of genus $g$  with a point $P\ne (0,1,0)$ of odd degree, and choose $P$ to be of   smallest odd degree $m$.  
  If $P$ is rational (that is, $m=1$) then $D(C)^{\text{\rm inf}} \supset \Z_{\geq g+1 }$.
  Otherwise $m\leq   g+1$ and $D(C)^{\text{\rm inf}} \supset \Z_{\geq  2g+2-m }$. 
   We deduce that $\mathcal E(C)$ is a subset of the odd integers   $\leq \max\{ 1, 2g-3\}$.
 \end{Proposition}

 In particular, this confirms ``half" of Bhargava's conjecture (hyperelliptic curves which have odd degree points will have an odd degree point of  degree $\leq g+1$). It remains to show that $100\%$ of such hyperelliptic curves have no points of odd degree   $<g$.
 
   \smallskip

  \subsection*{Strategy and  organization}
 
  The general strategy in our proofs is to insert irreducible polynomials for $x_1=x_1(t),\dots, x_m=x_m(t)$ of given degrees $d_1,\dots,d_m$ to obtain $f(\mathbf{x}(t))$ of degree $D$ (the maximal degree of any of the monomials in $f$), which is either irreducible, or is a given polynomial
  times an irreducible. This is guaranteed to happen for most inputs using Hilbert's irreducibility theorem.
  If $(d_1,\dots,d_m)=1$ then the roots of $f(\mathbf{x}(t))$ give rise to new points of degree $D$ on $C$.
  
If one does not get the same point  nor the same field too often, then one obtains a good lower bound for the number of fields containing new points that one can construct by this method
(which is Theorem \ref{thm: the count}). However for certain $f$ there can be a large number of solutions to 
 $f(\mathbf{x}(t))=g(t)$, and so we can only work with $f$ for which this does not happen (which we call \emph{the reasonable image hypothesis}); in practice we guarantee this by fixing some of the $x_i(t)$. 
 
In section \ref{sec: Define Exp(H)} we explain the reasonable image hypothesis, and define Exp$(C)$, and give the lower bound $\frac 1{2d^2}$. 
Section \ref{sec: ENT} shows how to construct certain exponent vectors more-or-less  in a given pre-determined direction, explains how to construct $D(H)^{\text{inf}}$ and $D(C)^{\text{inf}}$ leading to the proof of Theorem \ref{thm: C_f and H+}.

Section \ref{sec: HIT} puts Hilbert's Irreducibility Theorem in our context where we specialize an $m$-variable irreducible polynomial taking irreducible polynomials as entries, with an irreducible outcome, sometimes with extra conditions, and 
section \ref{sec: Non-sing points} develops the theory of non-singular solutions to our equation.

In section \ref{sec: OneVarPolys} we carefully make substitutions and keep track of how many irreducible polynomials we have created with given degree and coefficient size.
In section \ref{sec: New Fields} we determine how many new fields come as a result of the constructions in section \ref{sec: OneVarPolys}, and this is where the reasonable image hypothesis plays an important role.
 Section   \ref{sec: New Fields} contains the proofs of Theorem \ref{thm2} and Theorem \ref{thm: the count revisited}.

In section \ref{sec: Frob} we look at the special case that $f$ has points of given degree, leading to the proof of Theorem \ref{thm: A rational point}. In section \ref{ LB on Exp(H) } we make a  detailed look at improved lower bounds for Exp$(C)$, in particular for equations of special interest (like hyperelliptics and diagonal equations).

 In section \ref{sec: G(H)=G(C), curves} we look carefully at certain cases in which we can use local methods to show that $G(C)$ is must be divisible by a given prime, and therefore cases where we can fully determine $G(C)$ from $G(H)$. In particular, section \ref{sec: G(H)=G(C), curves} contains the proofs of Theorems \ref{thm: Hyperelliptic} and \ref{thm: PosProp}. 
 
 In section \ref{sec: Springer} we give our alternative proof of Springer's theorem \ref{thm: Springer} and the proof of Proposition \ref{prop: cubics}.  Section \ref{sec: Springer} also includes a discussion of odd degree points on hyperelliptic curves including Bhargava's conjecture and contains a proof of Proposition \ref{Prop: HyperI}.

 \subsection*{Acknowledgements}
 Many thanks to Manjul Bhargava, Pete Clark, Henri Darmon, Robert Lemke Oliver, Morena Porzio, Samir Siksek and  Frank Thorne for helpful remarks and/or useful email conversations. AG is grateful for support from NSERC (Canada). LB is grateful for the support of the U.S. National Science Foundation (DMS-2418835) and the Simons Foundation (MPS-TSM–00007992).

\section{The reasonable image hypothesis and Exp($C$)} \label{sec: Define Exp(H)}

We begin by stating a slightly more general version of Theorem \ref{thm: the count}:

\begin{Theorem} [Theorem \ref{thm: the count} revisited] \label{thm: the count revisited}
Let $C$ be an irreducible hypersurface given by a polynomial  $f$ in $m$ variables with integer coefficients, of degree $d$.
There exists a constant \emph{Exp}$(C)>0$ for which the following holds.

Fix $\epsilon>0$. If $D$ is divisible by $G(C)$ and is sufficiently large, and if $X$ is sufficiently large then 
\[
\#\{ \text{\rm Fields } K\in \mathcal F(C): [K:\mathbb Q]=D \text{ and {\rm  disc}}(K/\mathbb Q)  \leq X\} \gg X^{\frac 12\text{\rm Exp}(C)-\epsilon}.
\]
\end{Theorem}

We will define $\text{\rm Exp}(C)$ in terms of the Newton polygon $H_f$, and the \emph{reasonable image hypothesis} on repetitions in the image of $f$ with polynomial inputs under certain (explicitly given) restrictions. These definitions are complicated and will occupy the rest of this  section. One can sometimes obtain a better explicit result by taking into account further information about $f$.

\subsection{The reasonable image hypothesis} \label{sec: RIH I} A key technical issue is to determine how many different polynomial $m$-tuples substituted into $f$ give rise to the same polynomial value?  
For any polynomial $g(t)=\sum_i g_it^i$ we define $\| g(t)\|=\max_i |g_i|$.
If $\mathbf{x}(t)=(x_1(t),\dots ,x_m(t))$ then we define $\| \mathbf{x}(t) \|:= \max_{1\leq i\leq m} \| x_i(t) \| $, and  
$\deg \mathbf{x}(t):=\max_{1\leq i\leq m} \deg x_i(t)$.
\bigskip

\noindent \textbf{The reasonable image hypothesis (RIH).}  \textsl{Let   $f(x_1,\dots,x_m)\in \mathbb Z[t][x_1,\dots,x_m]$, homogeneous in $x_1,\dots, x_m$, have degree $d\geq 1$.   For any given irreducible polynomial $g(t)\in \mathbb Z[t]$, let $N_{f,D}(g,T)$ denote the number 
distinct $m$-tuples  $\mathbf{x}(t)\in \mathbb Z[t]^m$ with $\deg \mathbf{x}(t)\leq D$, given leading coefficients and $\|  \mathbf{x}(t) \|\leq T$ for which $f( \mathbf{x}(t))=g(t)$. Then
\[
N_{f,D}(g,T)\ll_f T^{o_f(D)} .
\]
}  

The reasonable image hypothesis holds whenever $m=1$ since then $f(x_1)-g(t)$ is a polynomial in $x_1$ of degree $d$ and so there are no more than $d$ solutions.
 
 The RIH holds in many situations, but not always. For example,
the polynomial $f(x,y)=x+y$ clearly fails the RIH since the number of solutions to 
$x(t)+y(t)=t^D+ct+1$ (assuming it is irreducible) with $\| x\|\leq T$ and say $\deg x, \deg y=D$ and given leading coefficients is $\gg T^{D}$.
One can create many such unreasonable images based on this   idea, for example whenever
$f(x_1,\dots,x_m)=g(x_1+x_2,x_3,\dots,x_m)$ for some $g(y_1,\dots,y_{m-1})\in \mathbb Z[y_1,\dots,y_{m-1}]$.

Another class of examples is given by polynomials $f$   for which $m> d^2$; these all fail the reasonable image hypothesis:
There are $\asymp T^{D}$ polynomials $x_i(t)$ of degree $D$, given leading coefficients with $\| x_i(t)\|\leq T $,
and so the number of $m$-tuples $(\mathbf{x}(t))$ is $\asymp T^{mD}$.
Now  
$f(\mathbf{x}(t))$ has degree $\leq dD$, with coefficients $\| f(\mathbf{x})\| \ll_f \| x(t)^d\|\ll (DT)^d\ll T^d$.
 The number of  polynomials $g(t)$ with degree $\leq dD$, bounded leading coefficient and $\| g\| \ll T^d$ is $\asymp_{d,D} T^{d^2D}$.
 Therefore if $m> d^2$ then some $g(t)$ must equal $f(\mathbf{x}(t))$ for $\gg  T^{(m-d^2)D}\gg T^{D}$ different choices of $\mathbf{x}(t)$.

 The following example further illustrates that the failure of the RIH can depend  on the coefficients of $f$
 and not just $H_f$: Suppose that  for some polynomials $f_1,f_2$ we have
   \[ f(x_1, \dots, x_m)= f_1(x_1, \dots, x_\ell)+ f_2(x_{\ell+1}, \dots, x_m).\]
     If  $f_1(y_1(t), \dots, y_\ell(t))=0$ (which depends on the coefficients of $f_1$) then for any given 
   $y_{\ell+1}(t),\cdots ,y_m(t)$ we have
   \[ f(a(t)y_1(t), \dots, a(t)y_\ell(t)), y_{\ell+1}(t),\cdots ,y_m(t))=g(t) \text{ where } g(t):=f_2(y_{\ell+1}(t),\cdots ,y_m(t))\]
   and we can take any $a(t)$ with $\deg a\leq D-\deg \mathbf{y}(y)$ and $\| a(t)\| \ll T/\max_i \| y_i(t)\|$, so that the RIH fails.

\subsection{Sub-images} \label{sec: subimage}
To get good lower bounds on the number of solutions our proofs need the reasonable image hypothesis to hold for the polynomial under consideration. We have just seen several examples where it fails; however for $f(x,y)=x+y$ the reasonable image hypothesis holds whenever we fix $y$ and vary only over polynomials $x$. More generally we proceed as follows:

If $\mathbf{y}(t)\in \mathbb Z[t]^m$ for which $f(\mathbf{y}(t))$ is irreducible and $J$ is a non-empty subset of $\{ 1,\dots, m\}$ then $f(\mathbf{x}_{J,\mathbf{y}}(t))$ is irreducible 
where we obtain $\mathbf{x}_{J,\mathbf{y}}(t)$ by  fixing $x_j(t)=y_j(t)$ for all $j\in \{ 1,\dots, m\}\setminus J$, and allowing the $x_j$ with $j\in J$ to be variables. Now let $\mathcal J_f$ denote those non-empty subsets $J$ of $\{ 1,\dots, m\}$ for which there exists  $\mathbf{y}(t)\in \mathbb Z[t]^m$ for which $f(\mathbf{y}(t))$ is irreducible and 
$f(\mathbf{x}_{J,\mathbf{y}}(t))$ satisfies the reasonable image hypothesis. 

We saw that every one-variable polynomial satisfies the RIH and so $\{ j \} \in \mathcal J_f$ for all $j\in \{ 1,\dots, m\}$.

There are many possibilities for subimages that we have not explored. Here we look at subvarieties of $C$ given by fixing
$x_j$ for each $j\not\in J$ (that is, intersecting the variety with a linear surface). But there are many other possible subvarieties, even other linear surfaces, for example considering whether  RIH holds when $x_1, x_2$ vary with $x_1+x_2$   fixed.  A broad consideration of such possibilities might lead to significant improvements to our explicit results (and in particular the value of Exp$(C)$ which we define below).
 
\subsection{Working with corners} \label{sec: Corners} For any  $\mathbf{r}\in \mathbb R_{\geq 0}^m\setminus \{ \mathbf{0}\}$ we select $ \mathbf{h}_r\in H$ to maximize $\mathbf{r} \cdot \mathbf{h}$.\footnote{There may be more than one possibility for $\mathbf{h}_r$; we can choose any of them.} For each $ \mathbf{h}\in H$ we define the sets 
 \[
R(\mathbf{h})=R_H(\mathbf{h}):= \{ \mathbf{e}\in \R_{> 0}^m:\  |\mathbf{e}|=1,  \mathbf{e}\cdot\mathbf{h}> \mathbf{e}\cdot\mathbf{h}' \text{ for all } \mathbf{h}'\in H \}
\]
so that $\overline{R(\mathbf{h})} = \{ \mathbf{e}\in \R_{\geq 0}^m:\  |\mathbf{e}|=1,  \mathbf{e}\cdot\mathbf{h}\geq  \mathbf{e}\cdot\mathbf{h}' \text{ for all } \mathbf{h}'\in H \} $. Now $\deg f=  \mathbf{h\cdot 1}$ for all $\mathbf{h}\in H_f$
 so we define $\deg H=d$  for all $H\in \mathcal H_{m,d}$ as $\mathbf{h\cdot 1}=d$ for all $\mathbf{h}\in H$.

For each non-empty subset  $J$ of $\{ 1,\dots, m\}$ we define 
\[
H_J = \{ \mathbf{h}_J: h\in H\}
\]
where $\mathbf{h}_J$ is given by the $j$th coordinate of the vector $\mathbf{h}$ for each $j\in J$ (that is, $\mathbf{h}_J$ is the vector $\mathbf{h}$ restricted to the co-ordinates in $J$). If $\mathbf{h}_J\in H_J $ then we define
 \[
R(\mathbf{h}_J)=R_{H}(\mathbf{h}_J):= \{ \mathbf{e}\in \R_{> 0}^J:\  |\mathbf{e}|=1,  \mathbf{e}\cdot\mathbf{h}> \mathbf{e}\cdot\mathbf{h}' \text{ for all } \mathbf{h}'\in H \}
\]
where the vectors in $\R_{> 0}^J$ are supported only on the $j$th co-ordinates with $j\in J$.
Let $\deg H_J:= \max_{\mathbf{h}\in H_J} \mathbf{h\cdot 1}$; and if $H=H_f$ then $\deg f_J:=\deg H_J$.
 Finally we define
 \[
  \text{\rm Exp}_J(H) :=   \max_{\substack{\mathbf{h}_r\in H \\ \mathbf{h}_e\in H_J}} 
  \max_{\substack{\mathbf{r}\in \overline{R(\mathbf{h}_r)} \\ \mathbf{e}\in \overline{R(\mathbf{h}_e)}}}  
  \frac{\mathbf{e} \cdot \mathbf{r}}{ ( \mathbf{e}\cdot \mathbf{h}_{\mathbf{e}}) (\mathbf{r}\cdot \mathbf{h}_{\mathbf{r}})}.
 \]
 
 Fix $j\in \{ 1,\dots,m\}$. Let $\mathbf{e}=\mathbf{r}$ be the vector with a $1$ in the $j$th co-ordinate and $0$ elsewhere and so $\mathbf{e}\cdot \mathbf{h} = \mathbf{h}_j$ 
 which implies that  $\mathbf{e}\cdot \mathbf{h}_{\mathbf{e}} = \max_{\mathbf{h}\in H} h_j=\deg_{x_j} f$. Therefore
  \[
  \text{\rm Exp}_{\{ j\}}(H) \geq \frac 1{(\deg_{x_j} f)^2} .
  \]

 Let $\mathcal J$ be a set of non-empty subsets $J$ of $\{ 1,\dots, m\}$ and define
    \[
   \text{Exp}(H, \mathcal J) := \max_{J\in \mathcal J}  \text{\rm Exp}_J(H) .
  \]
  The definitions imply that if $I\subset J$  then  
  $ \text{\rm Exp}_I(H) \leq  \text{\rm Exp}_J(H)$.  Therefore if  
  $\{ 1,\dots,m\} \in \mathcal J$ then
  \footnote{In section \ref{sec: Geometric} we restrict the possible $\mathbf{h}$ and $\mathbf{e}$ that give the maximum to an accessible finite set.}
\[
 \text{Exp}(H, \mathcal J) =   \text{\rm Exp}_{\{ 1,\dots,m\} }(H)= \max_{\mathbf{h}\in H} \max_{\mathbf{e}\in \overline{R(\mathbf{h})}}\frac 1{  (\mathbf{e}\cdot \mathbf{h})^2} .
 \]
This implies   $ \text{Exp}(H, \mathcal J) \geq \frac 1{  (\mathbf{e}\cdot \mathbf{h})^2}  \geq \frac 1{  \| \mathbf{h}\|^2} $ for any $\mathbf{h}\in H, \mathbf{e}\in \overline{R(\mathbf{h})}$. In particular,  taking $\mathbf{e}=\mathbf{1}/\sqrt{m}$, this implies that $ \text{Exp}(H, \mathcal J) \geq \frac  m{d^2}$. We discuss lower bounds on $\text{Exp}(H, \mathcal J)$ in more detail in section \ref{ LB on Exp(H) }.

 Finally if $C$ be an irreducible hypersurface given by a polynomial  $f$ in $m$ variables with integer coefficients, of degree $d$ then for Theorem \ref{thm: the count revisited} we define
 \[
\boxed{  \text{Exp}(C) = \text{Exp}(C_f) :=  \text{Exp}(H_f, \mathcal J_f)}
 \]
 where $\mathcal J_f$ was defined in section \ref{sec: subimage}.
  We have seen that if $j\in \{ 1,\dots,m\}$ then $j\in \mathcal J_f$. Hence
 \[
 \text{Exp}(C_f) \geq \max_{j\in \{ 1,\dots,m\}}  \text{\rm Exp}_{\{ j\}}(H) \geq \frac 1{(\deg_{x_j} f)^2} \geq \frac 1{d^2}
 \]
  which is how  ``Theorem  \ref{thm: the count}, revisited''  implies Theorem  \ref{thm: the count}.

\subsection{A more geometric perspective} \label{sec: Geometric} If $\{ 1,\dots,m\}\in \mathcal J$ then 
$\text{Exp}(H, \mathcal J) =  \frac 1{  (\mathbf{u}\cdot \mathbf{h})^2}$ where we select  $\mathbf{h}\in H, \mathbf{u}\in \overline{R(\mathbf{h})}$ for which
$|\mathbf{u}\cdot \mathbf{h}|$ is minimized. The following lemma restricts our search space for $\mathbf{h}$ and $\mathbf{u}$ to a simple finite set of possibilities.

\begin{Lemma} \label{lem: sortmin} Suppose that $|\mathbf{u}\cdot \mathbf{h}|$ is minimized, over all ${\mathbf{h}\in H, \mathbf{u}\in \overline{R(\mathbf{h})}}$ at $\mathbf{h}=\mathbf{h}_1, \mathbf{u}=\mathbf{u}_1$.
Let $\mathcal L_1$ be the set of linear forms given by 
\[
\mathcal L_1:=\{ (\mathbf{h}_1-\mathbf{h})\cdot \mathbf{x}: \mathbf{h}\in H\} \cup \{ \mathbf{e}_j\cdot \mathbf{x}: 1\leq j\leq m\} 
\]
where $\mathbf{e}_j$ is the unit vector in the $j$th coordinate direction. There exists a subset  $S$ of $\mathcal L_1$ which generates a subspace of $\mathbb R^m$ of dimension $m-1$ such that $\mathbf{u}_1$ generates the 1-dimensional null space of $S$ and 
$\mathbf{\ell}_{\mathbf{x}=\mathbf{u}_1}>0$ for all $\mathbf{\ell}\in \mathcal L_1 \setminus S$.
\end{Lemma}

To search for the optimal $\mathbf{h}$ and $\mathbf{u}$ we begin by applying Lemma \ref{lem: sortmin} for each $\mathbf{h}_1\in H$.  There are only finitely many subsets of the finite set $\mathcal L_1$ and only finitely many of these that generate a subspace of $\mathbb R^m$ of dimension $m-1$. In these cases we determine the unit vector (up to sign) that 
generates the null space, and then verify whether there is a choice of sign which satisfies the criteria in the lemma.
Finally we test which of these gives the largest value of $|\mathbf{u}\cdot \mathbf{h}|$.

\begin{proof} Write $\mathbf{u}=\mathbf{u}_1$.
Now $\mathbf{h}_1\cdot \mathbf{u}_1$ is a local minimum under the conditions that 
each $u_i\geq 0$, that $\| \mathbf{u} \| =1$ and  that  $\mathbf{h}_1\cdot \mathbf{u}\geq \mathbf{h}\cdot \mathbf{u}$ for all $\mathbf{h}\in H$ (that is, $\mathbf{\ell}_{\mathbf{x}=\mathbf{u}_1}\geq 0$ for all $\mathbf{\ell}\in \mathcal L_1$).

We will suppose that $\mathbf{u}$ satisfies various boundary conditions, $\mathbf{h}_i\cdot \mathbf{u}=\mathbf{h}_1\cdot \mathbf{u}$ for all $i\in I$ (and no others), and $u_j=\mathbf{e}_j\cdot \mathbf{u}=0$ for all $j\in J$ and no others.
Let $V$ be the subspace of $\mathbb R^m$ generated by 
\[
\langle \{ h_i-h_1: i\in I\} \cup \{  \mathbf{e}_j: j\in J\}  \rangle 
 \]
 so that $\mathbf{u}\in N(V)$, the null space of $V$.  
 
 For any    $\mathbf{v}\in N(V)$ there exists $\eta>0$ such that $\mathbf{u}+\eta \mathbf{v}$ satisfies the same inequalities as 
 $\mathbf{u}$ with respect to the boundaries, 
 $\mathbf{h}_1\cdot (\mathbf{u}+\eta \mathbf{v})> \mathbf{h}_i\cdot (\mathbf{u}+\eta \mathbf{v})$ for all $i\not\in I$, and $(\mathbf{u}+\eta \mathbf{v})_j>0$ for any $j\not\in J$.
 We must have  $\mathbf{h}_1\cdot \mathbf{u}\leq \mathbf{h}_1\cdot \frac{(\mathbf{u}+\eta \mathbf{v})}{\|  \mathbf{u}+\eta \mathbf{v}\|}$ which is equivalent (after squaring) to 
 \[
 (\mathbf{h}_1\cdot \mathbf{u})^2 ( 2 \eta \mathbf{u} \cdot  \mathbf{v} +\eta^2 \|\mathbf{v}\|^2) \leq 
 2\eta (\mathbf{h}_1\cdot \mathbf{u})(\mathbf{h}_1\cdot \mathbf{v})+\eta^2(\mathbf{h}_1\cdot \mathbf{v})^2
 \]
 Therefore taking arbitrarily small $\eta>0$  implies that 
  \[
 (\mathbf{h}_1\cdot \mathbf{u}) (\mathbf{u} \cdot  \mathbf{v})  \leq 
  (\mathbf{h}_1\cdot \mathbf{v}) 
 \]
 as $\mathbf{h}_1\cdot \mathbf{u}>0$,
 and this is also true when we change $\mathbf{v}$ to $-\mathbf{v}$ so that 
 \[
 (\mathbf{h}_1\cdot \mathbf{u}) (\mathbf{u} \cdot  \mathbf{v})  =
  (\mathbf{h}_1\cdot \mathbf{v}) \text{ for all } \mathbf{v}\in N(V).
 \]
Going back above we also have 
\[ (\mathbf{h}_1\cdot \mathbf{u})^2  \|\mathbf{v}\|^2 \leq  (\mathbf{h}_1\cdot \mathbf{v})^2
=(\mathbf{h}_1\cdot \mathbf{u})^2 (\mathbf{u} \cdot  \mathbf{v})^2
\]
so that $ \|\mathbf{v}\|^2\leq( \mathbf{u} \cdot  \mathbf{v})^2 \leq \|\mathbf{u}\|^2 \|\mathbf{v}\|^2=\|\mathbf{v}\|^2$. Therefore
$ |\mathbf{u} \cdot  \mathbf{v}|= \|\mathbf{u}\| \, \|\mathbf{v}\|$ so that $\mathbf{v}$ is a scalar multiple of $\mathbf{u}$.
This implies that $N(v)$ is a 1-dimensional space, which is generated by $S$ in the statement of the result.
  \end{proof}

\section{Elementary number theory}  \label{sec: ENT}

In this section we  collect together the  elementary  number theory we will use to exploit our constructions. 
Let $H$ be a given \emph{Newton polytope}, which is defined to be a finite subset of $\mathcal M_{m,d}$ such that
for each $\mathbf{h}\in H$ there exists $\mathbf{v}\in \mathbb Z_{\geq 0}^m$ for which
\[
\mathbf{h} \cdot \mathbf{v} > \mathbf{i} \cdot \mathbf{v} \text{ for all } \mathbf{i}\in H, \mathbf{i}\ne \mathbf{h}.
\]
We will  define a set $D(H)^{\text{inf}}$ that contains all the elements of $D(C)^{\text{inf}}$ which can be obtained without any specific knowledge of the hypersurface $C=C_f$ other than that its Newton polygon  $H_f$ is $H$. 
The next result provides a typical construction.

\subsection{Elements of $D(H)^{\text{inf}}$ constructed from a given element of $H_f$}
 
 \begin{Proposition} \label{prop: euclid}  Fix $\mathbf{h}\in H$ and $\mathbf{r}\in R(\mathbf{h})$.
  If $D$ is divisible by $g(\mathbf{h})$ and is sufficiently large then there exists
  $\mathbf{d}\in \mathbb Z_{\geq 1}^m$   with $(d_1,\dots, d_m)= 1$ 
  and each $d_j=\frac{Dr_j}{\mathbf{r}\cdot \mathbf{h}}+O( (\log D)^2)$
  for which $D=\mathbf{d}\cdot \mathbf{h}$ and  $D> \mathbf{d}\cdot \mathbf{h}'$ for all $\mathbf{h}'\in H$.
 \end{Proposition}

We will prove this using a tool from sieve theory:
Let $J(D)$ be Jacobsthal's function, the minimal integer such that every interval of this length contains an integer coprime to $D$. Iwaniec \cite{Iw} proved that we can take $J(D)\ll (\omega \log \omega)^2\ll (\log D)^2$ where $\omega$ denotes the number of prime factors of $D$. Our path involves proving the following:

 \begin{Lemma} \label{lem: euclid} Fix $r_1,\dots,r_m>0$  and integers $i_1,\dots,i_m\geq 1$ with $m\geq 2$ . If $D$ is divisible by $\gcd_j i_j$ and is sufficiently large then there exist integers
$d_1,\dots, d_m\geq 1$ with $(d_1,\dots, d_m)= 1$ and $D=d_1i_i+\cdots +d_mi_m$ where
each 
\[
\bigg|d_\ell -\frac  {r_\ell D}{i_1r_1+\cdots+i_mr_m}\bigg|\leq   (i_1+i_2)J(D) 1_{ \ell=1, 2} +  \sum_{j=2}^m i_j  
 \]
\end{Lemma}

\begin{proof}[Proof of Proposition \ref{prop: euclid}]
Since  $\mathbf{r}\in R(\mathbf{h})$ we have $ \mathbf{r}\cdot \mathbf{h}> \mathbf{r}\cdot \mathbf{h}'$ for all $\mathbf{h}'\in H$. The strict inequality implies that this is also true for all $\mathbf{r'}$ in a small ball $B$ around $\mathbf{r}$; in particular we may assume that $\mathbf{r}\in \mathbb R_{> 0}^m$.  

We now apply Lemma \ref{lem: euclid}  with $\mathbf{i}=\mathbf{h}$ so that $\| \mathbf{r} - c_D \mathbf{d}\|\ll \frac{(\log D)^2}D$ where $c_D=\frac {\mathbf{r}\cdot \mathbf{h}}D$. If $D$ is sufficiently large then $c_D \mathbf{d}\in B$, 
so that $D=\mathbf{h} \cdot \mathbf{d} > \mathbf{j} \cdot \mathbf{d}$ for all $ \mathbf{j}\in H, \mathbf{j}\ne \mathbf{h}$.
\end{proof}

We prove Lemma \ref{lem: euclid} by induction on $m$ beginning with the $m=2$ case:

\begin{Lemma} \label{lem: euclid2} If $(i_1,i_2)=1$  with $i_1,i_2>0$  then for all sufficiently $D$ there exists 
integers $d_1,d_2>0$ with $(d_1,d_2)=1$ for which $D=i_1d_1+i_2d_2$ with 
\[
\bigg|d_1-\frac {r_1D}{i_1r_1+i_2r_2}\bigg| \leq i_2J(D) \text{ and }\bigg |d_2-\frac {r_2D}{i_1r_1+i_2r_2}\bigg|\leq i_1J(D).
\]
\end{Lemma}

\begin{proof}
Let $N$ be the nearest integer to $\frac {r_2D}{i_1r_1+i_2r_2}$ with $N\equiv D/{i_2} \pmod {i_1}$, so 
$|N-\frac {r_2D}{i_1r_1+i_2r_2}|\leq \frac {i_1}2$.  Now let $M=\frac{D-{i_2}N}{i_1} $ so that  $D={i_1}M+{i_2}N$
and $|M-\frac {r_1D}{i_1r_1+i_2r_2}| =\frac {i_2}{i_1} |N-\frac {r_2D}{i_1r_1+i_2r_2}| \leq \frac {i_2}2$.

We now select $k$ as small as possible for which $m=M-k{i_2}, n=N+k{i_1}$ with $(m,n)=1$.
Now $(m,n)$ divides $D$, so we need to ensure that for all primes $p$ dividing $D$, $p$ does not divide $m$ or $n$.
Now $p$ does not divide at least one of ${i_1}$ and ${i_2}$, say ${i_2}$, so $p$ does not divide $m$ provided $k\not\equiv M/{i_2} \pmod p$ (and if $p$ divides ${i_2}$ then $k\not\equiv -N/{i_1} \pmod p$). So we need to sieve out from integers $k$ one residue class for each prime $p$ dividing $D$, and so $|k|\leq \frac{ J(D)}2$.  
The result follows with $d_1=m$ and $d_2=n$. 
\end{proof}

\begin{proof}[Proof of Lemma \ref{lem: euclid}]
Suppose it is true when $gcd_j i_j=1$. Now if $gcd_j i_j=g$ then $g|D$ so write $D'=D/g$ and each $i_j'=i_j/g$.
Then $D=gD'=g(d_1i'_i+\cdots +d_mi'_m)=d_1i_i+\cdots +d_mi_m$ and the approximations stays good.
 
So henceforth assume that $gcd_j i_j=1$. We proceed by induction on $m\geq 2$. The result is given by lemma \ref{lem: euclid2} for $m=2$. Let $b_1= i_2J(D), b_2= i_1J(D)$ and $b_\ell=i_{\ell-1}$ for all $\ell\geq 3$.

Now write $s_m:=i_1r_1+\cdots+i_mr_m$ and for $m\geq 3$ let $g=(D,i_1,\dots,i_{m-1})$ so that $g|d_m$.

We  select $d_m$ divisible by $g$ with $d_m$ as close as possible to $\frac  {r_m D}{s_m}$  so that 
\[
\bigg|d_m-\frac  {r_mD}{s_m}\bigg|\leq g\leq i_{m-1}=b_m.
\]
 Let $D'=D-d_mi_m$ so that 
$|D'-\frac{s_{m-1}}{s_m}D|\leq i_mg $.
We can now proceed with the induction hypothesis so that 
$D'=d_1i_i+\cdots +d_{m-1}i_{m-1}$ with $(d_1,\dots,d_{m-1})=1$ where, for $1\leq \ell \leq m-1$,
\[
\bigg|d_\ell -\frac  {r_\ell D}{s_{m}}\bigg| \leq 
\bigg|d_\ell -\frac  {r_\ell D'}{s_{m-1}}\bigg| + \frac  {r_\ell }{s_{m-1}} \bigg| D'-\frac  {s_{m-1}D}{s_{m}} \bigg|
\]
\[
\leq b_\ell + \sum_{j=\max\{ 3,\ell\}}^{m-1} i_j +  \frac  {r_\ell i_mg }{s_{m-1}} \leq b_\ell + \sum_{j=\max\{ 3,\ell\}}^{m} i_j 
\]
since $gr_\ell\leq i_\ell r_\ell\leq s_{m-1}$. The result follows.
  \end{proof}

\subsection{Defining $D(H)^{\text{inf}}$}  \label{sec: Construct D(H)inf} 
 For each $\mathbf{h}\in H$ and any integer $n\geq 0$ let
\[
S_{\mathbf{h},j}(n):= \{ \mathbf{h}\cdot \mathbf{d}-n:\,  \mathbf{d}\in \Z_{\geq n}^m, (d_1,\dots,d_m)=1,   d_j\geq n+\deg_{x_j} f  \text{ and } \mathbf{h}\cdot \mathbf{d}\geq \mathbf{i}\cdot \mathbf{d} \text{ for all } \mathbf{i}\in H_f\},
\]
$S_{\mathbf{h}}(n)=\bigcup_j S_{\mathbf{h},j}(n)$
 and   $S_{\mathbf{h}}=S_{\mathbf{h}}(0)$.

We construct $D(H)^{\text{inf}}$ in stages: We begin with
\[
D(H)_0 = \bigcup_{\mathbf{h}\in H} S_{\mathbf{h}}.  
\]
Now given $D(H)_k$ for some $k\geq 0$ we define
\[
D(H)_{k+1} = \bigcup_{\substack{\mathbf{h}\in H \\ n\in \text{Frob}(D(H)_k)}} S_{\mathbf{h}}(n),
\]
so that $D(H)_0 \subseteq D(H)_1\subseteq \cdots $.

By Proposition \ref{prop: euclid}, there exists $n_{\mathbf{h}}>0$ for which $g(\mathbf{h})\Z_{\geq n_{\mathbf{h}}} \subset S_{\mathbf{h}}$ for each
{\bf h}$\in H$.  Elementary number theory then yields that there exists $N_H>0$ for which
$G(H) \Z_{\geq N_{\mathbf{h}}} \subset \text{Frob}(D(H)_0)$.  Then there exists $m_{\mathbf{h}}>0$ such that
$G(H) \Z_{\geq m_{\mathbf{h}}} \subset  \bigcup_{n\in \text{Frob}(D(H)_0)} S_{\mathbf{h}}(n)$  for each
{\bf h}$\in H$. Hence  $D(H)_1 = G(H) \Z_{\geq 1} \setminus \mathcal E_1(H)$ for some finite set $\mathcal E_1(H)\subset G(H) \Z_{\geq 1} $.
 Now we keep on iterating and so $\mathcal E_1(H)\supseteq \mathcal E_2(H)\supseteq\cdots $
Since $\mathcal E_1(H)$ is finite there exists some $k$ with $\mathcal E_k(H)=\mathcal E_{k+1}(H)$ and then 
$\mathcal E_j(H)=\mathcal E_k(H)$ for all $j\geq k$, so we define $\mathcal E(H):=\mathcal E_k(H)$. 

Finally we define $D(H)^{\text{inf}}= G(H) \Z_{\geq 1} \setminus \mathcal E(H)$.


\subsection{Construction of  $D(C)^{\text{inf}}$} Given only $H=H_f$ one can determine $D(H)^{\text{inf}}$ as in section \ref{sec: Construct D(H)inf} and we know that $D(H)^{\text{inf}}\subset D(C)^{\text{inf}}$.  If $G(C_f)\ne G(H_f)$ then   there exist new points $P_1,\dots,P_r$ degrees $n_1,\dots,n_r\in D(C_f)^{\text{inf}}$, not divisible by $G(H_f)$ for which $G(C_f)=$gcd$(G(H_f),n_1,\dots,n_r)$. To construct $D(C)^{\text{inf}}$ we let
\[
D(C)_0=D(H)_0 \cup \{P_1,\dots,P_r\} \text{ and then } D(C)_{k+1} = \bigcup_{\substack{\mathbf{h}\in H \\ n\in \text{Frob}(D(C)_k)}} S_{\mathbf{h}}(n),
\]
 for all $k\geq 0$. Then $D(C)^{\text{inf}} = \lim_{k\to \infty} D(C)_k$ (as we did for $D(H)^{\text{inf}}$) and 
 $D(C)^{\text{inf}}= G(C) \Z_{\geq 1} \setminus \mathcal E(C)$.
 Theorem \ref{thm: C_f and H+} follows immediately from this construction.



\section{Hilbert's Irreducibility Theorem}  \label{sec: HIT}

 Hilbert's Irreducibility Theorem implies that if $f(x,y,z)\in K[x,y,z]$ is irreducible then the polynomial $f(x(t),y(t),z(t))\in K[t]$ is irreducible for almost all choices of 
 $x(t),y(t),z(t)\in K[t]$. We now make this precise in the form needed.
 
\subsection{Thin sets} Throughout let $K$ be a field of characteristic 0. In section 9.1 of   \cite{Serre2}, Serre defined a set  $\Omega\subset K^n$ or $ \mathbb{P}_n(K)$  to be \emph{thin} if there is an algebraic variety $Z$ over $K$ and a morphism $\pi \colon Z \to   K^n$ or $ \mathbb{P}_n(K)$, respectively, such that $\Omega \subset \pi(Z(K))$ and 
the fiber of $\pi$ over the generic point is finite and $\pi$ has no rational section over $K$ 
(that is, $\pi$ is quasi-finite and has no section in a neighborhood of the generic point).

 He went on to classify thin sets and showed in section 9.2 of \cite{Serre2} (see Proposition 1 and the following comment) that  if $F(X; T_1,\dots ,T_n)$ is an irreducible polynomial over $K(T)$ then 
\[
\{  t= (t_1, \dots, t_n):\ F(X, t) \text{ is reducible over } K[X]\}
\]
 is a subset of a thin set $\Omega$ in $\mathbb{P}_n(K) \simeq \mathbb{P}_n(\mathcal{O}_K)$, where $\mathcal{O}_K$ be the ring of integers of $K$. (Moreover, Serre's Proposition 2 implies that $F(X, t)$ has the same Galois group as 
 $F(X, T)$ outside of a thin set.)

Now  let $K$ be a number field of degree $d$ and    $||x||=\text{max}_\sigma |\sigma x|$ for $x\in \mathcal{O}_K$, where $\sigma$ runs over the different embeddings of $K$ in $\mathbb{C}$.

\begin{Lemma}[Hilbert's Irreducibility Theorem] \label{lem: HITI} Let $f(X,T_1,...T_n) \in \mathcal{O}_K[X,T_1, \dots T_n]$ be an irreducible polynomial in $K[X,T_1,...,T_n]$, and $e_1,\dots,e_n \in \mathbb{R^+}$. 
Then
\begin{align*}
\#\{ (a_1,...a_n)\in \mathcal{O}_K^n: \text{Each } ||a_i|| \leq T^{e_i} \text{ and } f(X,a_1,...a_n) \text{ reducible in } K[X]\} & \\
\ll T^{(e_1+ \dots + e_n -\emph{min}(e_i)/2)d} & \log T.
\end{align*}
\end{Lemma}

\begin{proof}   As remarked above, Serre showed in section 9.2 of  \cite{Serre2}  that 
\[
M:=\{ (a_1, \dots a_n) \in   \mathcal{O}_K^n: f(X,a_1,...a_n) \text{ is reducible in } K[X]\}
\]
 is a thin set.

In Theorem 1 of section 13.1 of  \cite{Serre2}, Serre observes that a theorem of Cohen \cite{Coh} is easily modified to establish
 a uniform bound on the number of points of a thin subset $M$ of $\mathcal{O}_K$ inside a box of diameter $N$: For any given $(c^{(1)} \dots c^{(n)})$ we have
 \[
  \# \{ (x^{(1)} \dots x^{(n)}) \in M:\ \text{max}_{1\leq i\leq n} ||x^{(i)}-c^{(i)} || \leq N\} = O(N^{(n-1/2)d}\log N).
 \] 
 For given $e_i \in \mathbb{R^+}$ let $e_0=\min_{1\leq i\leq n} e_i$. We can bound 
  \[
M(\mathbf{e}):=  \# \{ (x^{(1)} \dots x^{(n)}) \in M:\ \text{max}_{1\leq i\leq n} ||x^{(i)} || \leq T^{e_i}\} 
 \] 
by covering the region in $dn$-dimensional cubes of side length $N=2T^{e_0}$. If $T$ is sufficiently large then we will need no more than
 \[
 T^{d(e_1-e_0)} \cdots T^{d(e_n-e_0)} \text{ such boxes},
 \]
 and so 
 \[
 M(\mathbf{e})\ll T^{d(e_1-e_0)} \cdots T^{d(e_n-e_0)} \cdot T^{(n-1/2)de_0}\log T=  T^{(e_1+ \dots + e_n -e_0/2)d}\log T.
 \qedhere
 \]
\end{proof}

\section{Non-singular solutions to an equation} \label{sec: Non-sing points}

Throughout the rest of this paper we will suppose that 
\[
 f(x_1,\dots,x_m)=\sum_{(i_1,\dots,i_m)\in I_f} a_{i_1,\dots,i_m} x_1^{i_1}\cdots x_m^{i_m}\in \Z[x_1,\dots,x_m]
 \] 
is irreducible and homogenous of degree $d$, where $a_{\mathbf{i}}\ne 0$  for all $\mathbf{i}\in I_f$. 

An algebraic solution $\gamma=(\gamma_1,\cdots,\gamma_m)$ to $f(\gamma)=0$ is \emph{singular} if 
$\frac{\partial f(\mathbf{x}) }{\partial x_j}\big|_{\mathbf{x}=\gamma}=0$ for each $j$.  Now
$d\cdot f = \sum_j x_j \frac{\partial f }{\partial x_j}$ and so if $\gamma$ is non-singular then there are at least two values of $j$ for which $\frac{\partial f(\mathbf{x}) }{\partial x_j}\big|_{\mathbf{x}=\gamma}\ne 0$. In particular, we can analogously determine that $f$ is non-singular whenever we de-projectivize $f$.

A polynomial solution $\mathbf{y}(t)=(y_1(t),\dots,y_m(t))\in \Z[t]^m$ to 
$f(\mathbf{y}(t))\equiv 0 \pmod {P(t)}$ is \emph{singular} if 
$\frac{\partial f(\mathbf{x}) }{\partial x_j}\big|_{\mathbf{x}=\gamma}\equiv 0 \pmod {P(t)}$ for each $j$. Here $P(t)$ is an irreducible polynomial. We remark that if   $P(t):=f(\mathbf{x}(t))$ is irreducible then it is non-singular else
 \[
 \frac{df(\mathbf{x}(t))}{dt} = \sum_{j=1}^m \frac{dx_j(t)}{dt} \cdot \frac{\partial f(\mathbf{x}) }{\partial x_j}\bigg|_{\mathbf{x}=\mathbf{x}(t)}\equiv 0 \pmod {P(t)},
 \]
so that $P(t)$ divides $P'(t)$ which is impossible.

We now show that singular algebraic solutions and singular polynomial solutions are essentially the same thing:

If $f(\mathbf{y}(t))\equiv 0 \pmod {P(t)}$ is  singular then let   $\alpha$ be a root of $P(t)$, that is   $P(\alpha)=0$ and let $\gamma_j=y_j(\alpha)$ for each $j$.  Therefore  
$f( \gamma)=f(\mathbf{y}(t))\big|_{t=\alpha}$ which must equal $0$ as $f(\mathbf{y}(t))$ is a multiple of ${P(t)}$.
Analogously   $\partial_j f(\gamma)=0$ for each $j$ and so $\gamma$ is a singular algebraic point on
$f(x_1,\cdots,x_m)=0$.

On the other hand if $\gamma$ is a singular solution to $f(\gamma)=0$ then suppose $\gamma\in \mathbb Q(\alpha)$, the smallest such field,  where $\alpha$ has minimal (irreducible) polynomial $P(t)$. We can write each $\gamma_j=y_j(\alpha)$ for some polynomials $y_j(t)$.  Then $f(\mathbf{y}(t))\big|_{t=\alpha}=f( \gamma)=0$ and so $P(t)$ divides $f(\mathbf{y}(t))$.
Analogously   $\partial_j f(\mathbf{y}(t))  \equiv 0 \pmod {P(t)}$ for each $j$ and so $f(\mathbf{y}(t))$ is a singular polynomial solution to
$f(x_1,\cdots,x_m)=0$.

It is important to note that this is not a 1-to-1 correspondence since we can obtain different $P(t)$ for a given $\gamma$ by selecting $\alpha'\ne \alpha$ for which $\Q(\alpha')=\Q( \alpha)$; for example, $\alpha'=\frac{a\alpha+b}{c\alpha+d}$, an invertible transformation (in which case $P_\alpha(t)=P_{\alpha'}( \frac{at+b}{ct+d}) (ct+d)^{\deg f}$. We call $P_\alpha$ and $P_{\alpha'}$
\emph{equivalent}, and note that the roots of equivalent polynomials generate the same number fields.
 
Throughout we let    $\mathcal P_f$ denote the set of irreducible polynomials $P(t)\in \mathbb Z[t]$ for which there exists a non-singular solution 
$y_1(t),\dots,y_m(t)\in \mathbb Z[t]$ to   $f(\mathbf{y}(t))\equiv 0 \pmod {P(t)}$.

 We make a straightforward but useful observation:

\begin{Lemma} \label{lem: Sing at c=0}
If  $f(x_1,\dots,x_m)\in K[x_1,\dots,x_m]$ is homogenous of degree $d$ and 
$f(x_1,\dots,x_m)-c$ has a singularity at $\mathbf{x}=\mathbf{x}_0$ then $c=0$. 
\end{Lemma}

\begin{proof}  If $f(x_1,\dots,x_m)-c$ has a singularity at $\mathbf{x}=\mathbf{x}_0$ then all the partial derivatives vanish and so 
\[
\frac{\partial }{\partial x_i} f(x_1,\dots,x_m)\big|_{ \mathbf{x}=\mathbf{x}_0}= \frac{\partial }{\partial x_i} (f(x_1,\dots,x_m)-c)\big|_{ \mathbf{x}=\mathbf{x}_0} =0
\]
for each $i$.
But then 
\[
c = f(x_1,\dots,x_m)\big|_{ \mathbf{x}=\mathbf{x}_0}=\frac 1d \sum_{i=1}^m x_i \frac{\partial }{\partial x_i} f(x_1,\dots,x_m)\big|_{ \mathbf{x}=\mathbf{x}_0}=0. \qedhere
\]
\end{proof}

\subsection{The algebra of already constructed non-singular points}
    
 We begin with some basic algebra of the $P(t)\in \mathcal P_f$.

\begin{Proposition} [Hensel lifting for polynomial solutions] \label{Prop: Lifting} 
If $P(t)\in \mathcal P_f$ then $P(t)^e\in \mathcal P_f$ for every integer $e\geq 1$.
\end{Proposition}

\begin{proof}  We proceed by induction on $e\geq 2$, as the hypothesis gives the $e=1$ case.
Let $c_j(t)$ be a polynomial which is the inverse of $\frac{\partial f(\mathbf{x}) }{\partial x_j}\big|_{\mathbf{x}=\mathbf{y}(t)} \pmod P$.
Take the solution $\mathbf{y}_e(t)$ for $e$, so for some polynomial $Q_e(t)$ we have
\[
  f(\mathbf{y}_e) = P(t)^e Q_e(t)
\]
(where $\mathbf{y}_e\equiv \mathbf{y} \pmod P$), let $w_i(t)=y_i(t)$ for $i\ne j$,
and  $w_j(t)=y_j(t)-c_j(t)Q_e(t)P^e$.
Now by the multivariable Taylor expansion we have
\[
f(\mathbf{w}) \equiv f(\mathbf{y}_e)  -c_j(t)Q_e(t) P^e \cdot \frac{\partial f(\mathbf{x}) }{\partial x_j}\bigg|_{\mathbf{x}=\mathbf{y}_e(t)}
\equiv 0 \pmod {P^{e+1}},
\]
as required.
\end{proof}

\begin{Lemma} [Chinese Remainder Theorem for polynomial solutions]  \label{Prop: CRT}  
If $G(t)\in \mathbb Z[t]$ factors over $\mathbb Q[t]$ into coprime powers of irreducibles  in $\mathcal P_f$
then there exist $y_1(t),\dots,y_m(t)\in \mathbb Z[t]$ such that  $G(t)$ divides $f(\mathbf{y}(t))$.
 \end{Lemma}

\begin{proof}   
The result for $G(t)=P_1(t)^{e_1}\cdots P_k(t)^{e_k}$ with each $P_i(t)^{e_i}\in  \mathcal P_f$.
Given a solution $\mathbf{y}_j$ to $f(\mathbf{y}_j)\equiv 0\pmod{P_j(t)^{e_j}}$ for each $j$,
let $\mathbf{y}\equiv \mathbf{y}_j \pmod{P_j(t)^{e_j}}$ for each $j$, by the Chinese Remainder Theorem, and then 
$f(\mathbf{y})\equiv 0 \pmod{G(t)}$.  
\end{proof}

\begin{Lemma}[Minimal elements in a residue class for polynomial congruences] \label{lem: Small solutions}
If $G(t)\in \mathbb Z[t]$ for which there exist
$y_1(t),\dots,y_m(t)\in \mathbb Z[t]$ such that  $G(t)$ divides $f(\mathbf{y}(t))$ then  there exist
$u_1(t),\dots,u_m(t)\in \mathbb Z[t]$ with each $\deg u_i\leq \deg G-1$ such that 
\[
G(t) \text{ divides } f(u_1(t),\dots,u_m(t)).
\]
\end{Lemma}

 \begin{proof}  Suppose $G(t)$   divides $ f(y_1(t),\dots,y_m(t))$ for some $y_1(t),\dots,y_m(t)\in \mathbb Z[t]$, and now select the solution of smallest possible degree. Let $G=g_0t^D+\dots$ so that $G$ divides $ f(g_0y_1(t),\dots,g_0y_m(t))$.  If  
$\deg y_j\geq D$ for some $j$ then writing $y_j=c_0t^d+\dots$  with $d\geq D$ let
$v_j(t):=g_0y_j(t) - c_0t^{d-D}G(t)$ which has lower degree than $y_j$ and $v_i(t)=g_0y_i(t) $ if $i\ne j$ so that  $G(t)$ divides $f(v_1(t),\dots,v_m(t))$, contradicting the minimality of the degrees of the $y_j$. Hence we deduce that there are solutions with each $\deg u_i\leq D-1$.
\end{proof}

\begin{Lemma} \label{lem: quotient irred}
Let $f(x_1,\dots,x_m)\in \mathbb Z[x_1,\dots,x_m]$ be irreducible.
Suppose that $G(t)$ is a   product of elements of $\mathcal P_f$ so there exist a nonsingular solution $(u_1(t),\dots,u_m(t))\in \mathbb Z[t]^m$, with each $\deg u_i\leq \deg G-1$ for which  $G(t)$ divides $ f(u_1(t),\dots,u_m(t))$. Then 
\[
g(x_1,\dots ,x_m)[t] = f(x_1G(t)+u_1(t),\dots, x_mG(t)+u_m(t))/G(t)\in \mathbb Z[x_1,\dots,x_m,t]
\]
is irreducible,   $I_g$ is contained inside  the convex hull of $I_f$ and $H_g=H_f$.
\end{Lemma}

\begin{proof}  Let $y_i(t)=x_i(t)G(t)+u_i(t)$ for each so that $f(\mathbf{y})\equiv f(\mathbf{u})\equiv 0 \pmod{G(t)}$.  
Therefore $g(x_1,\dots ,x_m)[t]  \in \mathbb Z[x_1,\dots,x_m,t]$. Now
 \begin{equation}\label{eq: g written out}
 g(x_1,\dots ,x_m)[t] = \sum_{{\bf i} \in I_g}   c_{\bf i}(t) x_1^{i_1}\cdots x_m^{i_m}  \in \mathbb Z[x_1,\dots,x_m,t]
 \end{equation}

When we expand the polynomial $f$ to obtain $g$, we see that the jth co-ordinate is $\mathbb Z[t]$-linear in  $x_j$, which means that when we expand the  $x_1^{i_1}\cdots x_m^{i_m}$ term with $(i_1,\dots,i_m)\in I_f$, we get monomials
$x_1^{j_1}\cdots x_m^{j_m}$  with each $j_\ell\leq i_\ell$  including monomial with each $j_\ell\leq i_\ell$. Therefore
$I_g$ is contained inside  the convex hull of $I_f$ and $H_g=H_f$. In particular if $\mathbf{i} \in H_f= H_g$
then $b_{\mathbf{i}}(t)=G(t)^{d-1} c_{\mathbf{i}}$
by comparing \eqref{eq: f written out} with \eqref{eq: g written out}.

Finally we prove that $g(x_1,\dots ,x_m)[t]$ is irreducible:

 If $g(x_1,\dots ,x_m)[t]$ is reducible then
 then one of the factors must be a polynomial $h(t)$  in $t$ else we could factor $f(x_1,\dots,x_m)$ by substituting an integer $n$ in for $t$ such that $G(n)\ne 0$, and then making a linear change of variables. Therefore there exists an irreducible polynomial
 $Q(t)$ which divides $g(x_1,\dots ,x_m)[t]$.
 
 Now $Q(t)$ divides $g(x_1,\dots ,x_m)[t] $ as presented in \eqref{eq: g written out}
 so $Q(t)$ divides every $c_{\bf i}(t)$. In particular if $\sum_j i_j=d$ then $c_{\bf i}(t) = G(t)^d$ and so $Q(t)$ divides $G(t)$
 (and therefore $Q(t)\in \mathcal P_f$). 
 We now compute $g(x_1,\dots ,x_m)[t] \pmod G$:  By construction if $s\geq 1$ then $G(t)^{s-1}$ divides $ c_{\bf i}(t) $ when $\sum_j i_j=s$, so we can restrict our attention to $s=0$ and $1$. Therefore
 \[
 g(x_1,\dots ,x_m)[t]  \equiv  \frac{f(\mathbf{u}(t))}{G(t)} + \sum_{j=1}^m f_j(\mathbf{u}(t)) \cdot x_j \pmod {G(t)} \text{ where } f_j(\mathbf{u}(t)):= \frac{\partial f(\mathbf{x}) }{\partial x_j}\bigg|_{\mathbf{x}=\mathbf{u}(t)}.
 \]
 Now $Q(t)$ divides $ g(x_1,\dots ,x_m)[t] $ and so $Q(t)$ divides $f(\mathbf{u}(t))$ and each $f_j(\mathbf{u}(t))$ which contradicts that  $f(\mathbf{u}(t))$ is non-singular.
 \end{proof}

\section{Creating one variable irreducible polynomials} \label{sec: OneVarPolys}

\begin{Lemma} \label{lem: Pick n} Let $g(x_1,\dots,x_m)\in \mathbb Z[x_1,\dots,x_m]$. There  exist positive integers $n_1,\dots, n_m$ with $n_1+\cdots +n_m\leq m+\deg g$ such that $g(n_1,\dots,n_m)\ne 0$.
\end{Lemma}

 \begin{proof} We prove this by induction on $m$.  If $m=1$ then $g$ is a monomial so we can take $n_1=1$. 
For any $m>1$ we write $g$ as a polynomial in $\mathbb Z[y_1,\dots,y_{m-1}][y_m]$ which has degree $d$, say, in $y_m$.
The coefficient of $y_m^d$ is a polynomial of degree $\leq \deg g-d$ in $y_1,\dots,y_{m-1}$.
By the induction hypothesis there exist positive integers $n_1,\dots, n_{m-1}$ with $n_1+\cdots + n_{m-1}\leq m-1+\deg g-d$
such that the leading coefficient of $g$ is non-zero, and therefore there exists $n_m\leq d+1$ such that 
$g(n_1,\dots,n_m)\ne 0$.
\end{proof}

 \begin{Corollary}\label{lem: All D} 
 Let $f(x_1,\dots,x_m)\in \mathbb Z[x_1,\dots,x_m]$. For any given
 $\mathbf{d}=(d_1,\dots, d_m)\in \mathbb Z_{\geq 0}^m$ there exist positive integers $n_1,\dots, n_m$ with $n_1+\cdots +n_m\leq m+\deg f$ such that the leading term of $f(x_1(t),\dots,x_m(t))$, when each $x_j(t)=n_jt^{d_j}+\cdots$ is a polynomial of degree $d_j$, is of the form $Ct^D$ where  
$D:=\max_{\mathbf{h}\in H_f} \mathbf{h}\cdot \mathbf{d} $ and $C$ is a bounded   positive integer.
Here an upper bound on $C$ can be determined explicitly in terms of the polynomial $f$ and is independent of the choice of $d_i$'s.
\end{Corollary}

\begin{proof} Given   $\mathbf{d}=(d_1,\dots, d_m)\in \mathbb Z_{\geq 0}^m$ suppose that 
$x_j(t) = n_jt^{d_j}+\cdots$ is a polynomial of degree $d_j$. 
Let $D:=\max_{\mathbf{h}\in H_f} \mathbf{h}\cdot \mathbf{d} $ so that 
$f(x_1(t),\dots,x_m(t))$ is a polynomial of degree $D$ with leading coefficient given by $g_{\mathbf{d}}(n_1,\dots,n_m)$ where
\[
g_{\mathbf{d}}(y_1,\dots,y_m) = \sum_{\substack{ (i_1,\dots,i_m)\in I_f: \\ d_1i_1+\cdots +d_mi_m=D}} a_{i_1,\dots,i_m} y_1^{i_1}\cdots y_m^{i_m} \in \Z[y_1,\dots,y_m] .
\]
We can then select the $n_i$ as in Lemma \ref{lem: Pick n} to guarantee that the leading coefficient $g_{\mathbf{d}}(n_1,\dots,n_m)$ is non-zero.  There are only finitely many distinct polynomials $g_{\mathbf{d}}(\mathbf{y})$ since
there are just finitely many subsets of $I_f$. The result follows.
\end{proof}

\begin{Proposition} \label{prop: AllSubs} Given $\mathbf{d}=(d_1,\dots, d_m)\in \mathbb Z_{\geq 0}^m,\ \mathbf{e}\in \mathbb R_{>0}$ let 
\[
R:= \mathbf{e} \cdot  \mathbf{d},\ N:=\mathbf{1} \cdot  \mathbf{d} \text{ and } e_0=\min_{1\leq i\leq m} e_i.
\]
Select integers $n_1,\dots, n_m$ and $C:=g_{\mathbf{d}}(n_1,\dots,n_m)$ as in Corollary \ref{lem: All D}. 
There are $ 2^NT^R(1+O(1/T^{e_0/2}))$   $m$-tuples of irreducible polynomials
$\mathbf{x}(t)=(x_1(t),\dots,x_m(t))\in \Z[t]^m$ where $ x_i(t)$ has leading term $n_it^{d_i}$ and $\| x_i(t)\|\leq T^{e_i}$, for which $f(\mathbf{x}(t))$ is non-singular and has leading term $Ct^{D}$ with 
 $\| f(\mathbf{x}(t))\| \ll_f  T^E$ where  
\[
D:=\max_{\mathbf{h}\in H_f} \mathbf{h}\cdot \mathbf{d} \quad \text{ and }  \quad  E:=\max_{\mathbf{h}\in H_f} \mathbf{h}\cdot \mathbf{e}.
\]
\end{Proposition}

 \begin{proof}
Let  $x_i(t)=\sum_{j=0}^{d_i} x_{i,j} t^j$ for each $i$ where the $x_{i,j}$ are all variables except that $x_{i,d_i}=n_i$; we will now prove
\[
F(t,x_{1,0},\ldots,x_{m,d_m}):=f(x_1(t),\dots,x_m(t))
\]
is irreducible:
If $t=0$ and $x_{i,j}=0$ for all $ j\geq 1$ then 
\[
F(0,x_{1,0},0,\dots,0,x_{2,0},0\dots,0,x_{m,0},0\ldots,0)=
f(x_{1,0},\dots,x_{m,0})\]
 which is irreducible (as the $x_{j,0}$ are variables), and therefore  $F(\cdot)$ must be irreducible.
 
 The number of $f(x_1(t),\dots,x_m(t))$ constructed (irreducible or not) is given by 
 \[
\#\{ x_{i,j}\in \mathbb Z: |x_{i,j}|\leq T^{e_i},1\leq i\leq m, 0\leq j\leq d_i-1\}
= \prod_i (2T^{e_i}+O(1))^{d_i}=2^N  T^R (1+O( T^{-e_0})).
\] 
We need to show that the number of these for which 
 $f(x_1(t),\dots,x_m(t))$  has degree $<D$, or for which $f(x_1(t),\dots,x_m(t))$  is reducible,
 or for which some $x_j(t)$ is reducible, has much smaller order. As discussed in section \ref{sec: Non-sing points}, if $f(x_1(t),\dots,x_m(t))$  is irreducible
 then it is non-singular.

We note that  $f(\mathbf{x}(t))$     has leading term $Ct^{D}$  by Corollary \ref{lem: All D}. 

Since $F(t,x_{1,0},\ldots,x_{m,d_m})$ is irreducible we  apply Lemma \ref{lem: HITI} with   $n=N$, where the $T_\ell$ are given by the $x_{i,j}$, and so the  number of choices of the $x_i(t)$ for which  $f(x_1(t),\dots,x_m(t))$ is reducible is
\[
\ll \prod_i   T^{e_id_i} \cdot (\log T)/T^{e_0/2}=T^{R-e_0/2}.
\]
  
Finally we need to ensure that each $x_i(t)$ is irreducible. Bhargava \cite{Bh} proved van de Waerden's conjecture that 
 $\ll (T^{e_i})^{d_i-1}$ of the $x_i(t)$ are reducible for each $i$, and thus 
\[
\ll \sum_i T^{R-e_i}\ll \ T^{R-e_0}
\]
 of our choices of the 
set $x_1(t),\cdots,x_m(t)$. Our count for the number of $m$-tuples of irreducible polynomials
 for which  $f(x_1(t),\dots,x_m(t))$ is  non-singular and irreducible of degree $D$ now  follows.

For the second result we note that the sum of the absolute values of the coefficients of $x_i(t)$ is $\leq d_i T^{e_i}$, so the sum of the absolute values of the coefficients of $f(x_1(t),\dots,x_m(t))$ is
\[
\leq \sum_{(i_1,\dots,i_m)\in I_f} |a_{i_1,\dots,i_m}| \cdot (d_1 T^{e_1})^{i_1}\cdots (d_m T^{e_m})^{i_m}
\ll_f T^E . \qedhere
\] 
\end{proof}

\begin{Corollary} \label{cor: anothermainthm}
Let $f(x_1,\dots,x_m)\in \mathbb Z[x_1,\dots,x_m]$ be irreducible, let  $G(t)$ be a   product of elements of $\mathcal P_f$ and define $g(x_1,\dots ,x_m)[t] \in \mathbb Z[x_1,\dots,x_m,t]$ as in Lemma \ref{lem: quotient irred}.
Given integers $d_1,\dots, d_m\geq \deg G$ and reals $e_1,\dots, e_m> 0$ define $R,N,E,e_0$ as in Proposition \ref{prop: AllSubs}, and  
\[
R':=R-\bigg(\sum_i e_i\bigg)\deg G, D':=D-\deg G \text{ and } N':=N-m\deg G.
\]
Then there  are $ 2^{N'}T^{R'}(1+O(1/T^{e_0/2}))$   $m$-tuples of irreducible polynomials
$\{ x_i(t): \deg x_i=d_i\} \in \Z[t]^m$  with coefficients $\ll_G T^{e_i}$ for which  $g(x_1(t),\dots,x_m(t))[t]$ is non-singular and irreducible of degree $D'$.
  Moreover, the   coefficients of  $g(x_1(t),\dots,x_m(t))[t]$ are all $\ll_f  T^E$.  Here $D'\geq (d-1)\deg G$.
  \end{Corollary}

\begin{proof} 
We let each $x_i(t)=y_i(t)G(t)+u_i(t)$ with the $u_i$ chosen as in Lemma \ref{lem: quotient irred}, and each $y_i(t)$ has degree $d_i-\deg G$ with all coefficients integers of size $\leq T^{e_i}$.
Since $H_g=H_f$ the argument is then completely analogous to that of Proposition \ref{prop: AllSubs}. 
We note that $D=\sum_jd_ji_j\geq \sum_j i_j\deg G =d\deg G$.
\end{proof}

\section{Counting the number of distinct fields created} \label{sec: New Fields}

\begin{Lemma} \label{lem: Fields}
Suppose that $x_1(t),\dots,x_m(t)\in \mathbb Z[t]$ are irreducible of   degrees $d_1,\dots, d_m$ with 
$(d_1,\dots,d_m)=1$. If $\alpha$ is an algebraic number then 
\[
\mathbb Q(x_1(\alpha),\dots,x_m(\alpha))=\mathbb Q(\alpha).
\]
\end{Lemma}

\begin{proof}    Now  $[\mathbb Q(\alpha):\mathbb Q(x_j(\alpha)) ]=d_j$ since $x_j(t)$ is irreducible of degree $d_j$ and
 \[
\mathbb Q(\alpha) \supset \mathbb Q(x_1(\alpha),\dots,x_m(\alpha)) \supset \mathbb Q(x_j(\alpha))  
\]
so that $d:=[\mathbb Q(\alpha): \mathbb Q(x_1(\alpha),\dots,x_m(\alpha)) ]$ divides each $d_j$. But then
$d$ divides $(d_1,\dots,d_m)=1$ and so $d=1$; the result follows.
\end{proof}

\begin{Lemma} \label{lem: FieldsBd}
For any given number field $K$, the number of  irreducible polynomials $F(x)\in \mathbb Z[x]$ of degree $D$ with 
$\| F\| \ll T^E$, bounded leading coefficient and $\Q[x]/(F(x))\cong K$ is $\ll_D T^{E+o(1)}$.
\end{Lemma}

\begin{proof}  More generally, let  $K$ be a given number field of degree $D$.  Proposition 2.2 of \cite{OT20} states that the number of monic polynomials $f(x)\in\Z[x]$ for which $\Q[x]/(f(x))\cong K$ with $\| f\|\leq H$ is $\ll_D H(\log H)^{D-1}$. Now if $F(x)=\sum_{i=0}^D c_ix^i$ with $c_D=c$  then let $f(x)=x^D + \sum_{i=0}^{D-1} c^{D-1-i} c_ix^i$ so that 
$f(cx)=F(x)$ and therefore $\Q[x]/(F(x))\cong \Q[x]/(f(x))$ with $f$ monic. Thus we can apply Proposition 2.2 of \cite{OT20} to the $f(x)$ for each $c\leq C$.  The result follows taking $H\asymp c^D T^E\ll C^DT^E\ll_D T^E$. 
\end{proof}

\begin{Lemma}\label{lem: Bound on disc}
Let $g(t)$ be an irreducible polynomial in $K[t]$ of degree $D$. If $g(\alpha)=0$ then 
{\rm disc}$(K(\alpha)/K) \ll_D \| g\|^{2D-2}$.
\end{Lemma}

\begin{proof}   It is known that $\text{disc}(g)/\text{disc}(L(\alpha))\in \mathbb Z^2$, so
$\text{disc}(L(\alpha))\leq \text{disc}(g)$. Now $\text{disc}(g)$ is a homogenous polynomial in its coefficients of degree $2D-2$
and so $ \text{disc}(g)\ \ll_D \| f\|^{2D-2}$.
\end{proof}

\begin{proof}[Proof of Theorem \ref{thm2}]
Fix $\mathbf{h}\in H$ and $\mathbf{r}\in R(\mathbf{h})$.
If $D$ is divisible by $g(\mathbf{h})$ and is sufficiently large then we determine $d_1,\dots,d_m$ 
by Proposition \ref{prop: euclid} which can be used in Proposition \ref{prop: AllSubs} to show that 
$D\in D(C)$. If we fix all but one of the $x_i(t)$ in  Proposition \ref{prop: euclid} then we obtain the same polynomial no more than $d$ times.   Lemma \ref{lem: FieldsBd} then implies that  $D\in D(C)^{\text{\rm inf}}$.  This in turn implies that 
$G(C)$ divides $g(\mathbf{h})$.

There exist integers $D_1, \dots, D_n\in D(C)^{\text{\rm inf}}$ such that $G(C)=\text{gcd}(D_1, \dots, D_n)$.
  For any $b \pmod {g(\mathbf{h})}$ such that $G(C)\mid b$,   there exist positive integers $m_1,\dots ,m_k$ such that $B:= m_1D_1 +\dots + m_kD_k \equiv -b \pmod{D}$. Therefore there is a polynomial $G(t)$ which is a product of elements of $\mathcal P_f$ for which $\deg G=B$.
  Now suppose that $D\equiv b \pmod {g(\mathbf{h})}$ and is sufficiently large.  We use Proposition \ref{prop: euclid} to determine $d_1,\dots,d_m$  so that $D=\mathbf{d}\cdot \mathbf{h} +B$  and then use Corollary \ref{cor: anothermainthm} to show that $D\in D(C)$, obtaining that $D\in D(C)^{\text{\rm inf}}$ using Lemma \ref{lem: FieldsBd} as in the previous paragraph.
   \end{proof}

\begin{Proposition} \label{Prop: New dawn}
Let $C$ be an irreducible hypersurface given by a polynomial  $f$ in $m$ variables with integer coefficients, of degree $d$.
Suppose that $H=H_f$ and $J\in \mathcal J_f$, so that the  reasonable image hypothesis holds for $H_J$.
Let   $\mathbf{d}\in \mathbb Z_{\geq 0}^m$   and $\mathbf{e}\in [0,1]^m$.

There are $\gg T^{R_J-o(D))}$ distinct non-singular, irreducible polynomials  $f(\mathbf{x}(t))$   of degree $D$ with $\| f(\mathbf{x}(t))\| \ll_f  T^E$,
where each $x_i(t)$ is an  irreducible polynomial  of degree $d_i$ with $\| x_i(t)\|\leq T^{e_i}$.
where  
$D:=\max_{\mathbf{h}\in H_f} \mathbf{h}\cdot \mathbf{d},\ E:=\max_{\mathbf{h}\in H} \mathbf{h}\cdot \mathbf{e}$ and 
$R_J:= \mathbf{e}_J \cdot \mathbf{d}_J$.
\end{Proposition}

 \begin{proof}
 We   fix any one solution $\mathbf{y}(t)$ given by Proposition \ref{prop: AllSubs}, so that 
 \[
 F_J((x_j)_{j\in J},t):=f(x_1,\cdots,x_m)\big|_{x_i=y_i(t) \text{ for all } i\in \overline{J}}
 \] 
 is irreducible (as we know that there is a  specialization that is irreducible).  Proposition \ref{prop: AllSubs} applied to 
 $F_J((x_j)_{j\in J},t)$ then yields $\sim 2^{N_J}T^{R_J} $   $m$-tuples of irreducible polynomials
$\mathbf{x}(t)\in \Z[t]^m$ where $ x_i(t)$ has degree $d_i$ and $\| x_i(t)\|\leq T^{e_i}$
 for which $f(\mathbf{x}(t))$ is non-singular and irreducible of degree $D$ with 
 $\| f(\mathbf{x}(t))\| \ll_f  T^E$ where $N_J:=\mathbf{1}_J \cdot \mathbf{d}_J  $.
 
 The same polynomial $g(t)$ is given by $F_J((x_j)_{j\in J},t)$ no more than $T^{o(D)}$ times by the 
   reasonable image hypothesis (since  $\| \mathbf{x}(t)\|\leq T$ and each $d_j\leq D$ by hypothesis), and so we have $\gg T^{R_J-o(D)}$ \emph{distinct} non-singular, irreducible polynomials  $f(\mathbf{x}(t))$.
\end{proof}

By proceeding analogously  but now using Corollary \ref{cor: anothermainthm}
instead of Proposition \ref{prop: AllSubs}, we obtain the following:

\begin{Corollary} \label{cor: OneSub2} Assume the hypothesis of Proposition \ref{Prop: New dawn}, 
let  $G(t)$ be a   product of elements of $\mathcal P_f$ with $\mathbf{d}\in \mathbb Z_{\geq \deg G}^m$ and define $g(\mathbf{x})[t]$ as in Lemma \ref{lem: quotient irred}.
 
There are $\gg T^{R'_J-o(D)}$ distinct non-singular, irreducible polynomials  $g(\mathbf{x}(t))$   of degree $D'$ with
 $\| g(\mathbf{x}(t))\| \ll_f  T^E$,
where each $x_i(t)$ is an  irreducible polynomial  of degree $d_i$ with $\| x_i(t)\|\leq T^{e_i}$,
where  $D':=D-\deg G$ and $R'_J:=R_J-(\mathbf{e}_J\cdot \mathbf{1}_J)\deg G$.
\end{Corollary}

\begin{Proposition} \label{Prop: First explicit}
Let $C$ be an irreducible hypersurface given by a polynomial  $f$ in $m$ variables with integer coefficients, of degree $d$,
and let $G(t)$ be a   product of elements of $\mathcal P_f$ of degree $n$ (including the possibility $G=0$).
We assume hypotheses and notation of Proposition \ref{Prop: New dawn} and Corollary \ref{cor: OneSub2} so that $J\in \mathcal J_f$. If $D$ is an integer that can be written as 
\[
D:=\max_{\mathbf{h}\in H_f} \mathbf{h}\cdot \mathbf{d}
\]
where $\mathbf{d}\in \Z_{\geq n}^m$ with $(d_1,\dots,d_m)=1$ then 
\[
\#\{ \text{Fields } K\in \mathcal F(C): [K:\mathbb Q]=D-n \text{ and disc}(K/\mathbb Q)  \leq X\} \gg 
X^{ \frac{\mathbf{e}_J \cdot \mathbf{d}_J}{2DE} -O(\frac 1D)-o(1)} .
\]
\end{Proposition}

\begin{proof} We prove this first with $G=0$ so that $\deg G=0$.
In Proposition \ref{Prop: New dawn}  we constructed $\gg T^{R_J-o(D)}$ distinct, non-singular degree $D$ irreducible polynomials $f(\mathbf{x}(t))$ with coefficients all $ \ll T^E$, where  each $x_j(t)$ is irreducible. Then Lemma \ref{lem: FieldsBd} implies that these generate  
$\gg T^{R_J-o(D)}$ distinct fields $\Q[t]/(f(\mathbf{x}(t)))$ since $E\ll 1$.
 Any  root $\alpha$ of $f(\mathbf{x}(t))=0$ therefore generates a field of degree $D$,
and $\mathbb Q(\mathbf{x}(\alpha))=\mathbb Q(\alpha)$ by Lemma \ref{lem: Fields}; moreover
{\rm disc}$(\mathbb Q(\alpha)/\mathbb Q) \ll_D T^{2(D-1)E}$ by Lemma \ref{lem: Bound on disc}.
Therefore we have constructed $\gg X^{\frac{R_J-o(D)}{2(D-1)E}}$ fields of degree $D$ with discriminant 
$\leq X$ which contain a point on $f(\mathbf{x})=0$.

Now $E\ll 1$ so that $R_J=\mathbf{e}_J \cdot \mathbf{d}_J+O(1)\ll D$, so that $\frac{R_J-o(D)}{2(D-1)E} = \frac{\mathbf{e}_J \cdot \mathbf{d}_J}{2DE} -O(\frac 1D)-o(1)$.
The  result  for $G=0$ follows.

We follow the same route when $G\ne 0$ but using Corollary \ref{cor: OneSub2} in place of
Proposition \ref{Prop: New dawn}, noting that for a given $G$ we have 
$(\mathbf{e}_J\cdot \mathbf{1}_J)\deg G\ll 1$ and so $R'_J=R_J+O(1)$.
\end{proof}

\begin{Proposition} \label{Prop: Any J theorem}
Let $C$ be an irreducible hypersurface given by a polynomial  $f$ in $m$ variables with integer coefficients, of degree $d$,
with $H=H_f$ and suppose that $n\in \text{Frob}( D(C))$.
Assume that  $J\in \mathcal J_f$.
If $D$ is divisible by $g(\mathbf{h})$ where $\mathbf{h}\in H$ and is sufficiently large then
\[
\#\{ \text{Fields } K\in \mathcal F(C): [K:\mathbb Q]=D-n \text{ and disc}(K/\mathbb Q)  \leq X\} \gg 
X^{ \frac 12\text{\rm Exp}_J(H) -o_{D\to \infty}(1)} .
\]
\end{Proposition}

\begin{proof} If $n\ne 0$ then we can construct elements of $\mathcal P_f$ from points that lead to elements of $D(C)$, and therefore construct $G(t)$ as in lemma \ref{lem: Small solutions} of degree $n$, for any $n\in  \text{Frob}( D(C) )$.

Select $\mathbf{r}\in R(\mathbf{h})$.
There exists  $\mathbf{d}\in \mathbb Z_{\geq 1}^m$   with $(d_1,\dots, d_m)= 1$ 
  and each $d_j=\frac{Dr_j}{\mathbf{r}\cdot \mathbf{h}}+O( (\log D)^2)$
  for which $D=\mathbf{d}\cdot \mathbf{h}$ and  $D> \mathbf{d}\cdot \mathbf{h}'$ for all $\mathbf{h}'\in H$
by Proposition \ref{prop: euclid}. This also implies that each $d_j\geq \deg G$ (as $D$ is sufficiently large). Therefore
the exponent in Proposition \ref{Prop: First explicit} is $\frac 12$ of
\[
 \frac{\mathbf{e}_J \cdot \mathbf{d}_J}{DE} =  \frac{\mathbf{e}_J \cdot \mathbf{r}_J}{ 
 ( \mathbf{h}_r\cdot \mathbf{r} )( \mathbf{h}_e\cdot \mathbf{e})} +O\bigg( \frac{(\log D)^2}D \bigg)
\]
  where $\mathbf{h}_r\in H$ is selected so that $\mathbf{h}_r\cdot \mathbf{r}$ is maximal and 
   $\mathbf{h}_e\in H$ is selected so that $\mathbf{h}_e\cdot \mathbf{e}$ is maximal.
   Now to maximize this expression we evidently can take  $ \mathbf{e}= \mathbf{e}_J$ since any non-zero co-ordinate outside 
   $J$ makes the denominator larger without effecting the numerator.  We may also divide numerator and denominator through by $\|  \mathbf{e}\|$ so that $ \mathbf{e}\in R(\mathbf{h}_e)$ and again $\mathbf{h}_e\in H_J$. We want to maximize the above over all the possible choices for $ \mathbf{e},  \mathbf{r}, \mathbf{h}_e, \mathbf{h}_r$ but the sets $R(\mathbf{*})$ are open so the above is $\geq \text{\rm Exp}_J(H) -\epsilon$ if $D$ is sufficiently large. The result follows.
      \end{proof}

\begin{Corollary} \label{Cor: For Exp}
Let $C$ be an irreducible hypersurface given by a polynomial  $f$ in $m$ variables with integer coefficients, of degree $d$,
with $H=H_f$. If $D$ is divisible by $G(H)$ and is sufficiently large, and $X$ is sufficiently large then 
\[
\#\{ \text{Fields } K\in \mathcal F(C): [K:\mathbb Q]=D \text{ and disc}(K/\mathbb Q)  \leq X\} \gg 
X^{ \frac 12\text{\rm Exp}(C) -o_{D\to \infty}(1)} .
\]
\end{Corollary}      

 \begin{proof}   Fix $\mathbf{h}_0\in H$. Since $G(H)=\text{gcd}(g(\mathbf{h}): \mathbf{h}\in H)$, there exists integers $a_{\mathbf{h}}$ such that $G(H)=\sum_{\mathbf{h}\in H} a_{\mathbf{h}}g(\mathbf{h})$.
Select $A_{\mathbf{h}}$ to be the least non-negative residue of $a_{\mathbf{h}} \pmod {g(\mathbf{h_0})}$ so that 
$\sum_{\mathbf{h}\in H} A_{\mathbf{h}}g(\mathbf{h}) \equiv G(H) \pmod {g(\mathbf{h_0})}$. 

Next we apply Proposition \ref{Prop: Any J theorem} with $n=0$ and $J\in \mathcal J$, so if $\mathbf{h}\in H$ and $D_{\mathbf{h}}$ is sufficiently large with 
$D_{\mathbf{h}}\equiv g(\mathbf{h})   \pmod {g(\mathbf{h_0})}$ then $D_{\mathbf{h}}\in D(H) \subset D(C) $.
This implies that $n_1=\sum_{\mathbf{h}\in H} A_{\mathbf{h}}D_{\mathbf{h}}\in \text{Frob}( D(H) )  \subset \text{Frob}( D(C) )$, and 
$n_1\equiv G(H) \pmod {g(\mathbf{h_0})}$. Therefore $n_k:=kn_1 \in \text{Frob}( D(C) )$ and 
$n_k\equiv kG(H) \pmod {g(\mathbf{h_0})}$ for $0\leq k< g(\mathbf{h_0})/G(H)$.
   
 Now if $D$ is divisible by $G(H)$ and is sufficiently large then it can be written as $D'-n_k$
 where $D'$ is divisible by $g(\mathbf{h})$ with $\mathbf{h}\in H$, for some $k$.  The result then follows from applying Proposition \ref{Prop: Any J theorem} for every $J\in \mathcal J_f$ to get the exponent $ \frac 12\text{\rm Exp}(H_f,\mathcal J_f)-o(1)= \frac 12\text{\rm Exp}(C_f)-o(1)$.
\end{proof}       

We do not use any information about $C$ in the proof of Corollary \ref{Cor: For Exp} other than $H_f$ and $\mathcal J_f$.
This means that any new points on $C$ that might arise from the arithmetic of the particular choice of coefficients of $f$, with given values of $H$ and $\mathcal J$,
have not been included in the calculation. Thus, for any $f$ with given $H_f$ and $\mathcal J_f$ we prove the following:

 \begin{proof} [Proof of Theorem   \ref{thm: the count revisited}]
 The result holds for any sufficiently large $D$ that is divisible by $G(H)$, by Corollary \ref{Cor: For Exp}.
 Now $G(C)$ divides $G(H)$ and if $G(C)<G(H)$ then there exists $n_1,\dots,n_\ell \in D(C)$ for which
 $G(C)=\text{gcd}(G(H),n_1,\dots,n_\ell)$, by definition. 
 
 We now proceed analogously to the proof of Corollary \ref{Cor: For Exp}:  First find integers $a_i$ for which 
 $\sum_{i=0}^\ell a_i n_i = G(C)$ (where $n_0=G(H)$), select 
 $A_i$ to be the least non-negative residue of $a_i \pmod {G(H)}$. This implies there are integers
 $N_j\in D(C)$ with $N_j\equiv j G(C)  \pmod {G(H)}$, and therefore we can apply Proposition \ref{Prop: Any J theorem} for every $J\in \mathcal J_f$ to obtain our claim.
 \end{proof}

\section{Frobenius sets} \label{sec: Frob}

 \begin{Corollary} \label{cor: FrobSubset} 
 If there are non-singular $P_i(t)$ of degree $m_j$ in $ \mathcal P_f$ for $1\leq i\leq r$ then 
 \[
 d\deg_{x_j}f+ \text{\rm Frob}(d,(d-1)m_1,\dots , (d-1)m_r)\subset D(C_f)^{\text{\rm inf}}.
 \]
 \end{Corollary}

 \begin{proof} For any integers $e_1,\dots, e_r\geq 0$, let
 $G(t)=\prod_i P_i(t)^{e_j}$ so that  $n=\deg G=\sum_i e_im_i$. For any given $k\geq 0$ let $d_i=n+k$ for all $i\ne j$ and
 $d_j=n+\max\{ k, \deg_{x_j}f\}$. Since the largest co-ordinate of $d_j$ is the $j$th co-ordinate, we must have
 $h_j= \deg_{x_j}f$.  
By  Corollary \ref{cor: anothermainthm}  we have 
\[
(d-1)n+dk + \max\{ 0, \deg_{x_j}f-k\}\deg_{x_j}f= \max_{\mathbf{h}\in H_f} \mathbf{h}\cdot \mathbf{d}-n\in  D(C_f)^{\text{\rm inf}}.
\]
 Now the possible values of $n$ are the integers in the set $\text{\rm Frob}(m_1,\dots , m_r)$.
 The result follows from taking $k\geq \deg_{x_j}f$.
  \end{proof}

\begin{proof} [Proof of Theorem \ref{thm: A rational point}]
The hypothesis implies that we can take $m_1=r=1$ in Corollary \ref{cor: FrobSubset} 
so that $d\min_j \deg_{x_j}f+ \text{\rm Frob}(d,(d-1))\subset D(C_f)^{\text{\rm inf}}.$

We observe that if $a\equiv -j \pmod d$ for $1\leq j\leq d-1$ then $a\in \text{\rm Frob}(d-1,d)$ if and only if 
$a\geq (d-1)j$. Therefore the largest integer not in $\text{\rm Frob}(d-1,d)$ must be $(d-1)(d-1)-d=d^2-3d+1$. The result follows
as $ \deg_{x_j}f\leq d$.
\end{proof}

\subsection{Superelliptic hypersurfaces}  \label{sec: SuperHype} Let $f$ be the projective version of
\[
Y: ay^q=g(x_2,\dots,x_{m-1})
\]
where  $q$ divides $d:=\deg g$ so that $\deg f=d$. (Twisted) hyperelliptic  curves are the example $q=2,m=3$.
We projectivize to obtain
\[
f(x_1,\dots,x_m) = ax_0^q x_m^{d-q} - g(x_2/x_m,\dots,x_{m-1}/x_m) x_m^d.
\]
Evidently $(q,0,\cdots,0,d-q)\in H_f$.

If $\alpha$ is a non-singular point on $Y$ of degree $r$ then we obtain a a non-singular point $\beta$ on $C_f$ of degree $r$
with $\beta_m=1$, and corresponding polynomial $P_\beta(t)\in \mathcal P_f$ of degree $r$.  Using the Chinese Remainder Theorem with several such points we obtain $G(t)=\prod_\beta P_\beta(t)$ together with a non-singular solution
$f(u(t))\equiv 0 \pmod G(t)$ where $u_m(t)=1$.
 
 We now apply Corollary \ref{cor: anothermainthm} but with $d_m=0, d_j=\deg G$ for $2\leq j\leq m-1$, and then $d_1$ is any integer  with $d_1\geq \frac dq \deg G$, so that the first monomial has highest degree when we substitute in polynomials $x_j=x_j(t)$ of degree $d_j$. Taking each $e_i=1$  we deduce that  $D=qd_1$ and so $D'=qd_1-\deg G\in D(C_f)^{\text{\rm inf}}$.
 Writing $d_1=\frac dq \deg G+e$ for $e\geq 0$ we see that $(d-1)\deg G+q\mathbb Z_{\geq 0} \subset   D(C_f)^{\text{\rm inf}}$.
 
 Now if there is a non-singular rational point on our superelliptic hypersurface then   $t\in \mathcal P_f$ and taking $G(t)=t^e$ for any $e\geq 0$ we deduce that $\text{\rm Frob}(d-1,q)\subset D(C_f)^{\text{\rm inf}}$. The largest integer not in 
 $\text{\rm Frob}(d-1,q)$ is $(d-1)(q-1)-q=dq-2q-d+1$ and so every integer $\geq (d-2)(q-1)$ belongs to $D(C_f)^{\text{\rm inf}}$.
 If $q=2$ then $\mathcal E(C_f)\subset \{ 1,3,\dots, d-3\}$, which proves   Theorem \ref {thm: Hyperelliptic}(b).



\section{Lower bounds on $ \text{Exp}(C)$} \label{ LB on Exp(H) }

Suppose that $f$ satisfies the reasonable image hypothesis (RIH) so that $\{ 1,\dots,m\} \in \mathcal J_f$.
  In section \ref{sec: Corners} we saw that this implies that
  $ \text{Exp}(C_f) \geq \frac  m{d^2}$.
  
Now if $f$ doesn't satisfy RIH, but $f_J$ does,  then
we obtain a lower bound $\text{\rm Exp}_J(H)$ by letting $\mathbf{e}=\mathbf{1}_J/\sqrt{\# J}$ so that 
  $ \mathbf{e}\cdot \mathbf{h}_{\mathbf{e}} = \deg H_J / \sqrt{\# J}$, and then  
  \[
 \text{Exp}(C_f) \geq  \text{\rm Exp}_J(H) \geq   \max_{ \mathbf{h}_r\in H  } 
  \max_{ \mathbf{r}\in \overline{R(\mathbf{h}_r)}}  
  \frac{\mathbf{1}_J \cdot \mathbf{r}}{ (  \deg H_J) (\mathbf{r}\cdot \mathbf{h}_{\mathbf{r}})} \geq \frac{\# J}{ (  \deg H_J)^2}
   \geq \frac{\# J}{d^2}
  \]
 letting $\mathbf{r}=\mathbf{1}_J$. In particular if  $J=\{ j\}$ then 
 $\text{Exp}(C_f) \geq  \frac 1{(\deg_{x_j}f)^2}    \geq \frac 1{d^2}$ as we showed in  section \ref{sec: Corners}.
We now study various interesting special cases:

\subsection{Diagonal surfaces:}  \label{sec: diag}
If   $f=a_1x_1^d+\cdots+a_mx_m^d$ then   
\[
H_f=\{ \mathbf{h}_i:=(0,\dots,0,d,0,\dots,0), 1\leq i\leq m\}.
\]
 This gives that
 $\overline{R(\mathbf{h}_i)}=\{ \mathbf{x}\in \mathbb R_{\geq 0}^m: \|x\|=1 \text{ and }  x_i\geq x_j\text{ for all } j\}$.
 Therefore if $ \mathbf{x}\in \overline{R(\mathbf{h}_1)}, \mathbf{y}\in \overline{R(\mathbf{h}_2)} $ then each $x_iy_j\leq x_1y_2$ so that 
 \[
 \frac{ \mathbf{x}  \cdot   \mathbf{y}}{ ( \mathbf{x}  \cdot   \mathbf{h}_1)( \mathbf{y}  \cdot   \mathbf{h}_2) } =
 \frac{ \sum_{i=1}^m x_iy_i}{ dx_1\cdot dy_2}\leq \frac{  m x_1y_2}{ d^2x_1 y_2}= \frac m{d^2}
 \]
 and so  $\mathbf{Exp}(H) =\frac m{d^2}$  if $f$ satisfies the RIH. In fact the proof yields that we attain the bound only if each $x_i=x_1$ and each $y_j=y_2$ so that $\mathbf{x}=\mathbf{y}=\frac 1{\sqrt{m}} \mathbf{1}$.
 Recall that we  showed in section \ref{sec: RIH I}  that $f$ cannot satisfy the RIH if  $m> d^2$.

\subsection{Superelliptics:} Let $f(x_1,x_2,x_3)$ be the homogenization of  $y^q=g(x)$ where $\deg g=d>q>1$ and $g(0)\ne 0$.   Then
$H_f=\{(d,0,0), (0,0,d), (0,q,d-q)\}$.  Now $ \mathbf{Exp}_{\{2\}}(H)=\frac 1{q^2}$ taking $\mathbf{e}=(0,1,0)$.
If $\mathbf{r}=\mathbf{e}=(\frac q{(q^2+d^2)^{1/2}},\frac d{(q^2+d^2)^{1/2}},0)\in \overline{R((d,0,0))}$,  then
\[
  \mathbf{Exp}_{\{ 1,2\}}(H) \geq  \frac 1{q^2}+\frac 1{d^2},
 \]
 and we have $ \mathbf{Exp}_{\{ 1,2,3\}}(H) \geq  \frac 3{d^2}$.
 One can compare this to the exponents found in \cite{OT21}, \cite{Ke}, \cite{BK26} who count fields generated by points on elliptic, hyperelliptic, and superelliptic curves, respectively.
   
\subsection{A symmetric polynomial:} If  $f=x^ay^b+y^az^b+z^ax^b$ where $a>b\geq 1$ then
  $H=\{ (a,b,0), (0,a,b), (b,0,a)\}$.  Taking $\mathbf{r}=\mathbf{e}=\frac{1}{\sqrt{2}(a^2+ab+b^2)^{1/2}}(a,a+b,b)\in \overline{R(a,b,0)}$,  
  \[
 \mathbf{Exp}_{\{ 1,2, 3\}}(H) \geq  \frac 2{a^2+ab+b^2} > \frac 3{(a+b)^2}.
\]

\subsection{More symmetric polynomials:}  Let $G$ be a group of permutations on $m$ letters and let
$f=\sum_{\sigma\in G} a_{\sigma(1)}^{i_1}\cdots a_{\sigma(m)}^{i_m}$ with $i_1\geq \cdots \geq i_m=0$.
We saw earlier that $\mathbf{Exp}(H) \geq 1/(\sum\limits_{j=1}^n i_j^2)$.  We can instead try
$\frac 1{\sqrt{2i_1^2-2i_1i_2+i_2^2}}(i_1-i_2,i_1,0,\dots,0)\in \overline{R(i_1, \dots, i_n)}$
 and so  
 \[
 \mathbf{Exp}(H) \geq  \frac{2i_1^2-2i_1i_2+i_2^2}{i_1^4}\geq \frac 1{i_1^2}.
 \]

\section{Cases where $G(H)$ and $G(C)$ are equal for curves}   \label{sec: G(H)=G(C), curves}

\subsection{Hyperelliptic curves with index 2} \label{sec: Hyper index 2}
    
\begin{Lemma} \label{lem: no odd points} Suppose that the binary form $f(x,y)\in \Z[x,y]$ of degree $2m$  is irreducible mod $p$, and that prime $p\| d$. Then $C: dz^2=f(x,y)$ has no points in fields of odd degree.
\end{Lemma} 

 \begin{proof}    Suppose $u$ and $v$ are algebraic integers, and $w$ is an algebraic number such that $f(u,v)=dw^2$ where $K=\mathbb{Q}(u/v,w/v^m)$ is a field of odd degree.\footnote{As the triple  $u:v:w$ comes from the weighted projective space  $\mathcal A(1,m)$. Thus we begin with a point $f(U,1)=dW^2$ and then select $v$ from the ring of algebraic integers of $K=\Q(U,W)$ for which $u:=Uv$ is also an algebraic integer. Then $w=Wv^m$.}
 We are able to assume that $u$ and $v$ are algebraic integers by multiplying through by any common denominator.  Now  let $R$ be an ideal in the same ideal class as $Q=(u,v)$
 and which is coprime with $p$. Select algebraic integers $q$ and $r$ so that $Q\overline{Q}=(q)$ and $R\overline{Q}=(r)$. We now replace $u$ and $v$ by  $ur/q$ and $vr/q$ which are algebraic integers since $(r/q)=R/Q$ and $Q|u$ and $v$, so we may assume that $(u,v,p)=1$. 
 
 Let $P$ be any prime ideal dividing $p$. We now show that $f(u,v)\not\equiv 0 \pmod P$. Suppose otherwise that 
 $f(u,v)\equiv 0 \pmod P$.  Since $(u,v,p)=1$, $P$ does not divide at least one of $u$ and $v$, say $v$.
 Then there is a linear factor of $f(x,y)$ in the  reduction   $K/P$ and so has relative degree $2m$ dividing $[K:\Q]$, a contradiction.
 
 Therefore $P$ does not divide $dw^2$, and so if $P^r$ divides the denominator of $w$ exactly then $P^{2r}$ divides $(p)$ exactly; in other words $e_P$ is even,  But then $[K:\Q]=\sum_{P|(p)} e_P$ is even, again a contradiction.
 \end{proof}
 
 \begin{Corollary} \label{cor: aa no odd points} Suppose that the binary form $f(x,y)\in \Z[x,y]$ of degree $2m$  is irreducible. Then $C: dz^2=f(x,y)$ has no points in any fields of odd degree for almost all squarefree $d$.
\end{Corollary} 

\begin{proof}  Let $\mathcal P(f)$ be the set of primes $p$ for which $f(x,y)$ is   irreducible mod $p$.
Then $\mathcal P(f)$ contains a positive proportion (say $\delta$) of all primes by the Cebotarev density theorem.
Let $\mathcal D(f)$ be the set of   integers that are not divisible by any prime from $\mathcal P(f)$; by sieve theory we know that $\# \{ d\leq x: d\in \mathcal D(f)\} \ll x/(\log x)^\delta$.  

Now if $d$ is squarefree and $C: dz^2=f(x,y)$ has a point  of odd degree then $p\nmid d$ for all $p\in \mathcal P(f)$ by Lemma \ref{lem: no odd points}. Thus the number of such $d\leq x$ is $ \ll x/(\log x)^\delta$ and the result follows as there are $\gg x$ squarefree $d\leq x$.
\end{proof}

In our next   result we fix $d$ and vary $f$, rather than the other way around.

Let $\text{Poly}(B,2m)$ be the set of binary forms $f(x,y)\in \Z[x,y]$ of degree $2m$ with $\| f\| \leq B$.

 \begin{Corollary} \label{cor: aa no odd points} Let $d$ be an integer $\leq B$ and let $S_d:=\{ \text{Primes } p: p\|d\}$. 
 The  probability that  $C: dz^2=f(x,y)$ has  at least one point in any field of odd degree, for a binary form $f(x,y)$  randomly chosen  from $\text{Poly}(B,2m)$, is $\ll_m (1-\frac 1{2m})^{|S_d|}$.
\end{Corollary} 

 \begin{proof} Let $q:=\prod_{p\in S_d} p\leq d \leq B$.  If $C: dz^2=f(x,y)$ has a point in a field of odd degree then $f(x,y)$ must be reducible mod $p$ for every prime $p$ dividing $q$, by Lemma \ref{lem: no odd points}. We dissect the domain for the coefficients up into boxes with side length $q$:  There are $(2B/q+O(1))^{2m+1}\ll_m (B/q)^{2m+1}$, such boxes.
 
 There are $\frac 1\ell \sum_{d|\ell} \mu(d) p^{\ell /d}= \frac{p^{2m}}{2m} (1+O(p^{-m}))$  irreducible polynomials in $\F_p[t]$ of degree $\ell (=2m)$ which split completely in $\F_{p^\ell}$,   but remain irreducible in  $\F_{p^k}[t]$ for $1\leq k\leq \ell-1$.
Therefore the number of polynomials mod $q$ that are   reducible mod $p$ for every prime $p$ dividing $q$ is
 \[
q^{2m+1} \prod_{p|q} \bigg( 1 -  \frac{1}{2m} (1+O(p^{-m}) \bigg) \asymp q^{2m+1}  \bigg( 1 -  \frac{1}{2m} \bigg)^{|S_d|} .
 \]
 Multiplying together the above yields the claimed result.
  \end{proof}

 \begin{proof} [Proof of Theorem \ref{thm: Hyperelliptic}(a)] Let the degree   $d=2m$.
 Let $f(\cdot)$ be the multiplicative function with each $f(p^k)=1-\frac 1{2m}$. 
 By Corollary \ref{cor: aa no odd points}  the probability that a curve $dy^2=a_{2m}x^{2m}+\cdots +a_0$ has a point of odd order, where $0<|d|,|a_0|,\dots,|a_{2m}|<B$  is  
 \[
\ll_m  \frac 1B \sum_{d\leq B} f(d) \ll \exp\bigg(   \sum_{p\leq B} \frac{1-f(p)}p \bigg)\ll \frac 1{ (\log B)^{1/2m}}
 \]
 using Halasz's theorem to bound the mean value of a positive valued real multiplicative function. This $\to 0$ as $B\to \infty$.
  \end{proof}
  
  \begin{remark}   \label{rem: A challenge}
To prove the analogy of Theorem \ref{thm: Hyperelliptic} for sets $I^+$ with 
  $I^+=\{ (0,2,d-2)\} \cup \{ (i,0,d-i): i\in J\}$ where $0,d\in J \subset \{ 0,1,\dots,d\}$.
  we need to count irreducible   degree $d$ polynomials mod $p$, where the 
  coefficients of $t^j, j\not\in J$ equal zero. This can be done sufficiently well provided 
  $\# J\geq \frac{3d}4$ using \cite[Theorem 1.1]{Ha}.
  
  \end{remark}

\subsection{Curves which have no points of order divisible by $n$}

 Given $\alpha\in K$, a given number field, and $P$ a prime ideal of $K$, define
 $v_P(\alpha)$ to be the exact power of $P$ that divides $(\alpha)$, positive or negative, so that the fractional ideal
 $(\alpha)/P^{v_P(\alpha)}$ is not divisible by $P$ in either the numerator or denominator
 
 \begin{Lemma}\label{lem: HomoPts} 
 Suppose $\alpha,\beta\in K$  a given number field, and that $P$ is a prime ideal of $K$. Then there exist
 $\alpha_P,\beta_P,\gamma_P\in K$ with $\alpha:\beta:1=\alpha_P:\beta_P:\gamma_P$, such that 
 $\min\{ v_P(\alpha_P),v_P(\beta_P)\}=0$ and $v_P(\gamma_P)<0$ if and only if $\min\{ v_P(\alpha),v_P(\beta)\}>0$.
 \end{Lemma}
 
 \begin{proof} Let $v_P(\alpha)=u$  and $v_P(\beta)=v$ so we can write the ideals
 \[
 (\alpha) = P^u A \text{ and } (\beta) = P^v B
 \]
 with $u,v\in \Z$ where $v_P(A)=v_P(B)=0$.

  Let $R$ be an integral ideal in the same ideal class as $P$ but with $(R,P)=1$. 
Let $w:=\min\{ u,v\}$, $R\overline{R}=(r)$ and $P\overline{R}=(q)$ and then let
\[
 \alpha_P=\alpha (r/q)^w,\ \beta_P =\beta (r/q)^w,  \text{ and }  \gamma_P = (r/q)^w
 \]
 Then $v_P(\alpha_P)=u-w, v_P(\beta_P)=v-w, v_P(\gamma_P)=-w$ and all the claims follow.
 \end{proof}

 \begin{Proposition} \label{prop: PosProp=m} Fix integer $m\geq 2$ and prime $p$.
 Suppose that  $f(t)\in \Z[t]$ has degree $m$, is irreducible and remains so when reduced mod $p$.  
 Let $g(x,y)\in \Z[x,y]$ be a polynomial of degree $\leq m-1$ with $p\nmid g(0,0)$. Finally let
 \[
 F(x,y):=y^m f(x/y)-pg(x,y)\in \Z[x,y].
 \]
 If $F(x,y)=0$ has a $K$-rational point then  $m$   divides $[K:\Q]$. Moreover  
 \[
D(C_F)^{\text{\rm inf}} = m\Z_{\geq 1}.
\]
\end{Proposition}

 \begin{proof} Let $P$ be a prime ideal of $K$ dividing $(p)$ and suppose $P^{e_P}\| (p)$ (that is  $v_P(p)=e_P$).
 Write $f(x,y)=y^mf(x/y)=\sum_{i=0}^m f_i x^i y^{m-i}$ where $f_0,f_m\not\equiv 0 \mod p$ since $f$ is irreducible of degree $m$ when reduced mod $p$. Write $g(x,y)=\sum_{i=0}^{m-1} g_i(x,y)$ where $g_i$ is homogenous of degree $i$ and then let $G(x,y,z):=\sum_{i=0}^{m-1} g_i(x,y)z^{m-i}$ which is homogenous of degree $m$.

  Suppose that $F(\alpha,\beta)=0$ where $\alpha,\beta\in K$, so that 
  \[
  f(\alpha,\beta) = p\, g(\alpha,\beta).
  \]
  Applying Lemma \ref{lem: HomoPts} there exist
 $\alpha_P,\beta_P,\gamma_P\in K$ with $\min\{ v_P(\alpha_P),v_P(\beta_P)\}=0$  such that 
 \[
  f(\alpha_P,\beta_P) = p\, G(\alpha_P,\beta_P,\gamma_P).
 \]
 
 Since $\alpha_P$ and $\beta_P$ are $P$-integral, so $f(\alpha_P,\beta_P)$ is also $P$-integral, that is
 $v_P(f(\alpha_P,\beta_P))\geq 0$.
 
 Now suppose that $v_P(f(\alpha_P,\beta_P))\geq 1$, so that $f(\alpha_P,\beta_P) \equiv 0 \pmod P$.
 We claim that $v_P(\beta_P)=0$ for if not then  $v_P(\alpha_P)=0$ and  $ f(\alpha_P,\beta_P) \equiv f_m(\alpha_P)^m\not\equiv 0 \pmod P$, a contradiction.   Now $\alpha/\beta=\alpha_P/\beta_P$ and so
 $f(\alpha/\beta)=f(\alpha_P/\beta_P)=\beta_P^{-m} f(\alpha_P,\beta_P) \equiv 0 \pmod P$. Embedding this into $\overline{\F_p}$  we deduce that there exists a root $\rho$ of $f(t)$ in $\F_{p^m}$ such that $\rho=\alpha/\beta\in \F_p(\alpha/\beta)$, and so 
$ \F_p(\alpha/\beta)=\F_{p^n}$ where $m$ divides $n$.  However then $n$ must divide
$[\Q(\alpha/\beta):\Q]$ which divides $[K:\Q]$,  and so $m$ divides $[K:\Q]$.

Now suppose that $v_P(f(\alpha_P,\beta_P))=0$.
Then $v_P(\gamma_P)< 0$ else $v_P(G(\alpha_P,\beta_P,\gamma_P))\geq 0$ and so
 $v_P(f(\alpha_P,\beta_P))\geq e_P>0$, a contradiction. Lemma \ref{lem: HomoPts} implies that  
 \[ u:=\min\{ v_P(\alpha),v_P(\beta)\}>0\]
  so that $P^{iu}$ divides $g_i(\alpha,\beta)$ and so
 $g(\alpha,\beta)\equiv g_0(\alpha,\beta)=g(0,0)\not\equiv 0 \pmod P$, and therefore
 $v_P(f(\alpha,\beta))=e_P+v_P(g(\alpha,\beta))=e_P$.  
 Now $v_P(\alpha_P)=v_P(\alpha)-u$ and $v_P(\beta_P)=v_P(\beta)-u$ and so
 $v_P(f(\alpha_P,\beta_P))= e_P-mu=0$, and so $m$ divides $e_P$, the ramification index of $P$ in $K$. Therefore $m$ divides $\sum_{P|p} e_P=[K:\Q]$.
 
 Therefore we have proved that $m$ divides $[K:\Q]$ and so $D(C_F)^{\text{\rm inf}} \subseteq  m\Z_{\geq 1}$.

 For the second part of the argument let $x=x(t), y=y(t)$, polynomials of degree $k$. Typically $F(x(t),y(t))$ will be an irreducible polynomial of degree $mk$ in which case if $F(x(\alpha),y(\alpha))=0$, then $(x(\alpha),y(\alpha))\in C$ with 
 $[K:\mathbb Q]=mk$ for most pairs $x(t),y(t)$ by the Hilbert irreducibility theorem where $K=\mathbb Q(x(\alpha),y(\alpha))$.
 Therefore $mk\in D(C_F)^{\text{\rm inf}}$, that is, $m\Z_{\geq 1}\subseteq D(C_F)^{\text{\rm inf}}$. The result follows.
 \end{proof}
 
 Proposition \ref{prop: PosProp=m} for $m\geq 2$ and $p=2$, yields that if
 $F(x,y):=  f(x,y)-2g(x,y)$ where $f$ is homogenous of degree $m$ and 
 irreducible mod $2$, while $g(x,y)$ is a polynomial of degree $\leq m-1$ with $g(0,0)$ odd.
 Then $D(C_F)^{\text{\rm inf}} = m\Z_{\geq 1}$.
  
 \begin{proof}[Proof of   Theorem \ref{thm: PosProp}]
 We will determine what proportion of $F(x,y)=\sum_{i+j\leq m} f_{i,j} x^iy^j\in \Z[x,y]$
 with $f_{0,0}, f_{0,m},f_{m,0}\ne 0$  (so that $H_f=\{ (0,0,m), (0,m,0), (m,0,0)\}$) satisfy the 
 hypotheses of Proposition \ref{prop: PosProp=m} for some  $m\geq 2$ and $p=2$:
  
 We assume each $|f_{i,j}|\leq T$ for $T$ large, and then take
  $f(x,y)=\sum_{i+j=m} f_{i,j} x^iy^j$ with  $g(x,y)=\sum_{i+j<m} f_{i,j} x^iy^j$. 
  The hypothesis for $g$ requires $f_{i,j}$ even for all $0<i+j<m$ and $f_{0,0}\equiv 2 \pmod 4$; these happen for probability
  $(1/2)^{m(m+1)/2-1}\cdot 1/4=1/2^{(m^2+m+2)/2}$.  
  The hypothesis for $f(x,y)$ requires $f(x,1)$ to be an irreducible polynomial mod $2$ of degree $m$, which we saw happens 
  with probability $\frac 1m(1+O(2^{-m/2}))$. Therefore in total $\ll 1{m2^{(m^2+m+2)/2}}$ of polynomials $F$ satisfy the 
 hypotheses of Proposition \ref{prop: PosProp=m} (with $p=2$) and so have $D(C_F)^{\text{\rm inf}} = m\Z_{\geq 1}$.
 
 We have proved the result for $I_F= \mathcal M_{3,d}$. However the proof changes little if
 we remove $(i,j,k)$ from $I_F$ whenever $1\leq k\leq m-1$, which gives the complete result.
 \end{proof}

 \section{Odd degree points on low degree hypersurfaces and Springer's Theorem}  \label{sec: Springer}
 
\begin{Theorem}[Springer]  \label{Thm: Springer}
 Let $F(x_1, \dots, x_m)$ be a homogeneous quadratic, then $F(x_1, \dots, x_m)$ has a rational point if and only if it has a point in a number field of degree $d$ for any odd degree $d$.
\end{Theorem}

\begin{proof} First we show that if $F$ has an odd degree point then it has a rational point. Let $P$ be a point of lowest odd degree on $F$, and suppose $d>1$. Write the field as $\mathbb{Q}(\alpha)$ where $\alpha$ has minimal polynomial $g(t)$ of degree $d$. Now suppose $P=(x_1(\alpha), \dots, x_m(\alpha))$, where we can assume $D=\max_j \deg x_j(t) \leq d-1$ so that $F(x_1(t), \dots, x_m(t))$ has degree $2D$. Now $g(t)$ divides $F(x_1(t), \dots, x_m(t))$ and $h(t)=F(x_1(t), \dots, x_m(t))/g(t)$ has degree $2D-d$ which is odd and $<d$, so $h(t)$ must have an irreducible factor of odd degree $< d$ which contradicts the minimality of $d$, so the only possible minimal odd degree is $d=1$ and thus $F$ has a rational point.

Now, suppose $F(x_1, \dots, x_m)$ has a rational point $Q=(q_1,\dots,q_m)$.
For polynomials $x_1(t), \dots, x_m(t)$ of degree $d$, the polynomial $G(t)=F(tx_1(t)+q_1, \dots, tx_m(t)+q_m)/t$, generically gives degree $2d+1$ points for each $d\geq 0$.
\end{proof}

\begin{Proposition}[Coray]  \label{Prop: CubicForm}
A cubic form $F(x_1, \dots, x_m)$ has a rational point if and only if it has a degree $2$ point.
\end{Proposition}

\begin{proof} First we show that if $F$ has a point of degree 2    then it has a rational point. Let $P$ be a point of even degree. Write the field as $\mathbb{Q}(\alpha)$ where $\alpha$ has minimal polynomial $g(t)$ of degree $2$. Now suppose $P=(x_1(\alpha), \dots, x_m(\alpha))$, where   $D=\text{max } \text{deg } x_j(t) =1$ (to see this note that $D>0$ else $P$ is a rational point, and also by construction $D\leq \deg  g-1=1$). Then $F(x_1(t), \dots, x_m(t))$ has degree $3$ and is divisible by $g(t)$, so that $h(t)=F(x_1(t), \dots, x_m(t))/g(t)$ is a linear polynomial and, solving for $t$, thus $F$ has a rational point.

Now, suppose $F(x_1, \dots, x_m)$ has a rational point $Q=(q_1,\dots,q_m)$.
The polynomial $F(t)= F(a_1t+q_1, \dots, a_mt+q_m)/t$   has degree $2$ and is generically irreducible, so that 
$F$ has a point of degree 2.
\end{proof}

Therefore suppose that a cubic form $F$ has a degree $2$ point. By Proposition \ref{Prop: CubicForm} it has a rational point.
If that rational point is non-singular then $D(C_f)^{\text{\rm inf}}\supset \Z_{\geq 2}$ by Theorem \ref{thm: A rational point}, in particular it will have infinitely many non-singular points of degree $2$.

 \subsection{Odd degree points on hyperelliptic curves}
 
 Bhargava, Gross, and Wang \cite[Corollary 8]{BGW},  showed that for any fixed odd integer $m$ the proportion of hyperelliptic curves of genus $g$ that contain points of degree $m$, tends to $0$ as $g\to \infty$.
 So if a curve does have odd degree points, what degrees are they?
 
  \begin{Conjecture}[Bhargava]
$100\%$ of the locally soluble hyperelliptic curves which have odd degree points will have an odd degree point of lowest degree \emph{exactly} $g$ or $g+1$ (whichever is odd).
 \end{Conjecture}

 In this section we will prove Proposition \ref{Prop: HyperI}. Part (a)   establishes ``half'' of Bhargava's conjecture. It remains to show that 0\% of such curves have points of odd degree $<g$.
 
 Without loss of generality we may restrict our attention  to $C: y^2=f(x)$ where the leading coefficient of $f(x)$ is not a square, $f(x)$ is squarefree, and $f$ has even degree $2g+2$ (so that $C$ has genus $g$:  To see that recall (from  the introduction) that if $f(x)$ has odd degree $d$ then   there is a bi-rational transformation between the algebraic points of $y^2=f(x)$ and those of $Y^2=F(X)$ where $F(t)=t^{d+1}f(1/t)$   has (even) degree $d+1$ We could have begun by 
  translating $x\to x+m$ to ensure that $f(0)$ is not a square, so that the  leading coefficient of $F(x)$ is not a square.
   
  It is evident that  every positive even integer $2k$ is in $D(C)^{\infty}$, by substituting in a generic polynomial of degree $k$ for $y$ and an arbitrary constant $x$, in which case almost all of these substitutions will lead to an algebraic point of degree $2k$.

 \begin{proof} [Proof of Proposition \ref{Prop: HyperI}] As explained above we write $C$ as $y^2=f(x)$ where $f(x)$ has degree $2g+2$, and  the leading coefficient of $f$ is not a square.
 
 \emph{Suppose that $P$ is rational.} By a translation of $x$ we may suppose that $P=(0,y_0)$, so that 
 $f(x)=y_0^2+c_1x+ \cdots + c_d x^d$ for some $c_j\in \Q$ where $d=2g+2$, and we assume that $y_0\ne 0$.
 We select $y=y_0+y_1t+\dots+y_{m} t^m+u t^{m+1}$ for some $0\leq m\leq g$, where the $y_1,y_2,\dots,y_m$ are 
 selected so that the coefficient of $t^j$ in $y^2$ equals $c_j$; note that this is possible since $y_j$ does not appear in the coefficients of $t^i$ with $i<j$, and only in the monomial $2y_0y_j$ in the coefficient for $t^j$. Thus
 $2y_0y_1=c_1, 2y_2y_0=c_2-y_1^2$, etc.  Note that $u$ is a free variable.
 
 By construction $f(t)-y^2$ is divisible by $t^{m+1}$, and $(f(t)-y^2)/t^{m+1}$ has degree $2g+1-m$, which will be irreducible for most choices of $u$, and so  $2g+1-m\in D(C)^\infty$ for $0\leq m\leq g$, and therefore $[g+1, 2g+1] \cap \mathbb{Z}$ in  $D(C)^{\infty}$.
 
 Finally we can take $y=y_0+tw(t)$ with $\deg w=\ell \geq g$  so that $(y^2-f(t))/t$ has degree $2\ell+1$ and so all
 odd integers $\geq 2g+1$ belong to $D(C)^\infty$.
\medskip

\emph{Otherwise, suppose that  $P=(\alpha,\beta)$ be a point of  odd degree $m>1$}, so that $\beta$ is a root of $y^2-f(\alpha)$.   Then $\beta\in \Q(\alpha)$ else $2| [\Q(\alpha,\beta):\Q(\alpha)] $ which divides $m=[\Q(\alpha,\beta):\Q]$, a contradiction. 
Therefore we can write $\beta=h(\alpha)$ where $h(x)\in \Q[x]$ has degree $\leq m-1$.
Let $G(x)$ be the minimum polynomial for $\alpha$ which has degree $m$ so that 
 $h(\alpha)^2-f(\alpha)=\beta^2-f(\alpha)=0$ and therefore $G(x)$ divides $h(x)^2-f(x)$. 
 
  We write $h(x)^2-f(x)=G(x)R(x)$ where $R(x)$ is some polynomial in $\Q[x]$. The polynomial $h(x)^2-f(x)$ has even degree because both $h(x)^2$  and $f(x)$ have even degree, and there cannot be cancellation of the leading terms as the leading coefficient of $f$ is not a square.  Therefore  $G(x)R(x)$  has even degree and $G(x)$ has odd degree, so $R(x)$ has odd degree as well. Now $R(x)$ has an irreducible  factor $Q(x)$ of odd degree, so that  $h(\gamma)^2=f(\gamma)$ for any $\gamma$ for which $Q(\gamma)=0$, yielding a point of odd degree. Now
 $\deg R\geq \deg Q\geq m$ by minimality, so that 
 \begin{align*}  2m&\leq m+\deg R=\deg (h(x)^2-f(x)) \\
&=  \max\{ 2\deg h(x) , \deg f(x)\} \leq 2 \max\{ m-1 , g+1\} 
\end{align*}
and so $m\leq g+1$.

  Now write  $H(x)=h(x)+a(x)G(x)$ for an arbitrary polynomial $a(x)$ of degree $r\geq   g+1-m$.
 Then $H(\alpha)=\beta$ so that $H(\alpha)^2-f(\alpha)=\beta^2-f(\alpha)=0$ and therefore we can write $H(x)^2-f(x)=G(x)R(x)$ for some $R(x) \in \Q[x]$. Again there is no cancelation as the leading coefficient of $f$ is not a square, so
 $G(x)R(x)$ has odd degree.  For a typical $a(x)$, $R(x)$ is irreducible with 
\begin{align*}   \deg R&=\deg (H(x)^2-f(x))-m \\
&=  \max\{ 2\deg H(x) , \deg f(x)\}-m \\
&=   \max\{ 2r+m , 2g+2-m\} =2r+m.
\end{align*}
As we range over the possible values of $r$ we deduce that $D(C)^\infty\supset \Z_{\geq 2g+2-m}$.
  \smallskip
 
 Finally,  if   $1\not\in D(C)$ then $D(C)^{\infty}$ contains all the integers
 $\geq \max_{3\leq m\leq g+1}   2g+2 -m=2g-1$. Combining this with our first result we deduce that
$\mathcal E(C)$ is a subset of the odd integers in $\leq \max\{ 1, 2g-3\}$ as claimed.
 \end{proof}

 The simplest such question that we cannot answer is whether there  a curve $C: y^2=f(x)$ with $f$ of degree $8$ (genus $3$)  which has  no points in fields of degree 3?

 \bibliographystyle{plain}

\end{document}